\documentclass[11pt,a4paper]{article}
\usepackage[margin=3cm]{geometry}
\usepackage[utf8]{inputenc}
\usepackage{amsthm,amsmath,amssymb,color}
\usepackage{tikz,authblk,cite}
\usepackage{hyperref}
\usepackage{enumitem}
\usepackage{lineno}
\usepackage{thm-restate}
\hypersetup{
    colorlinks,
    linkcolor={red!60!black},
    citecolor={green!60!black},
    urlcolor={blue!60!black},
    pdftitle={Twin-width of random graphs},
    pdfauthor={Jungho Ahn, Debsoumya Chakraborti, Kevin Hendrey, Donggyu Kim, and Sang-il Oum}
}
\newtheorem{theorem}{Theorem}[section]
\newtheorem{proposition}[theorem]{Proposition}
\newtheorem{lemma}[theorem]{Lemma}
\newtheorem{claim}{Claim}[theorem]
\newtheorem{corollary}[theorem]{Corollary}
\newtheorem{observation}[theorem]{Observation}
\newenvironment{subproof}[1][\proofname]{
    
    \begin{proof}[#1]}{\end{proof}}
\newcommand\abs[1]{\lvert #1\rvert}
\newcommand{\tww}{\operatorname{tww}}

\newcommand{\E}{\mathrm{E}}
\newcommand{\Var}{\mathrm{Var}}
\newcommand{\DeltaR}{\Delta\!^R}
\newcommand{\rdeg}{\operatorname{rdeg}}

\makeatletter
\newcommand{\leqnomode}{\tagsleft@true\let\veqno\@@leqno}
\newcommand{\reqnomode}{\tagsleft@false\let\veqno\@@eqno}
\makeatother
\newcommand\ltag[1]{\tag*{\textbf{#1}}}

\begin{document}
\date{\today}
\title{Twin-width of random graphs%
\thanks{All the authors were supported by the Institute for Basic Science (IBS-R029-C1). Jungho Ahn was also supported by the KIAS Individual Grant (CG095301) at Korea Institute for Advanced Study. Debsoumya Chakraborti was also supported by the European Research Council (ERC) under the European Union Horizon 2020 research and innovation programme (grant agreement No.\ 947978).}}
\author[1]{Jungho Ahn}
\author[2]{Debsoumya Chakraborti}
\author[3]{Kevin Hendrey} 
\author[4,3]{Donggyu Kim}
\author[3,4]{Sang-il Oum}
\affil[1]{School of Computational Sciences, Korea Institute for Advanced Study (KIAS), Seoul,~Korea}
\affil[2]{Mathematics Institute, University of Warwick, Coventry, UK}
\affil[3]{Discrete Mathematics Group, Institute for Basic Science (IBS), Daejeon,~Korea}
\affil[4]{Department of Mathematical Sciences, KAIST, Daejeon,~Korea}
\affil[ ]{\small\textit{Email addresses:} 
\texttt{junghoahn@kias.re.kr},
\texttt{debsoumya.chakraborti@warwick.ac.uk},
\texttt{kevinhendrey@ibs.re.kr},
\texttt{donggyu@kaist.ac.kr},
\texttt{sangil@ibs.re.kr}}
\maketitle

\begin{abstract}
    We investigate the twin-width of the Erd\H{o}s-R\'enyi random graph $G(n,p)$.
    We unveil a surprising behavior of this parameter by showing the existence of a constant $p^*\approx0.4$ such that with high probability, when $p^*\le p\le1-p^*$, the twin-width is asymptotically $2p(1-p)n$, whereas, when $0<p<p^*$ or $1>p>1-p^*$, the twin-width is significantly higher than $2p(1-p)n$.
    In addition, we show that the twin-width of $G(n,1/2)$ is concentrated around $n/2-\sqrt{3n\log n}/2$ within an interval of length $o(\sqrt{n\log n})$.
    For the sparse random graph, we show that with high probability, the twin-width of $G(n,p)$ is $\Theta(n\sqrt{p})$ when $(726\ln n)/n\le p\le1/2$.
\end{abstract}

\section{Introduction}\label{sec:intro}

Twin-width is a graph parameter recently introduced by Bonnet, Kim, Thomass\'{e}, and Watrigant~\cite{twin-width1}.
Typically, twin-width is defined in terms of contractions on trigraphs; see Section~\ref{sec:prelim} for the formal definition.
Equivalently, the twin-width of a graph, denoted by $\tww(G)$, is the minimum integer~$d$ such that there is a way to move from the partition of the vertex set into parts having a single vertex each to the trivial partition having only one part consisting of all vertices by iteratively merging a pair of parts such that at every stage of the process, every part has both an edge and a non-edge to at most~$d$ other parts.
Such a sequence of partitions is equivalent to a $d$-contraction sequence.
Graphs of twin-width~$0$ are precisely cographs, which are graphs that can be built from copies of $K_1$ by taking disjoint unions and complements.

Twin-width is receiving a lot of attention not only because many interesting classes of graphs have bounded twin-width but also because in every class of graphs of bounded twin-width, any property expressible in first-order logic can be decided in linear time as long as the corresponding $d$-contraction sequence is given as an input~\cite{twin-width1}.
Furthermore, every class of graphs of bounded twin-width admits a function~$f$ such that the chromatic number of every graph $G$ in the class is bounded from above by $f(\omega(G))$, where $\omega(G)$ is the clique number of $G$~\cite{twin-width3}, which unifies and also simplifies known theorems on the $\chi$-boundedness of certain graph classes.
Graphs of bounded tree-width, graphs of bounded clique-width, planar graphs, 
graphs embeddable on a fixed surface, and all proper minor-closed graph classes are proved to have bounded twin-width~\cite{twin-width1}.
Twin-width has generated a huge amount of interest~\cite{twin-width1,twin-width2,twin-width3,twin-width4,twin-width6,GPT2022,BKRTW2021,BH2021,DGJOR2021,ST2021,BNOST2021,SS2021,modelcounting,chibounded}.

We investigate the twin-width of the Erd\H{o}s-R\'enyi random graph $G(n,p)$, which is a graph on a vertex set $\{1,2,\ldots,n\}$ chosen randomly such that each unordered pair of vertices is adjacent with probability $p$ independently at random.
Since the twin-width is invariant under taking the complement graph, we may focus on $G(n,p)$ with $p\le1/2$, because $G(n,p)$ and $G(n,1-p)$ have the same twin-width distribution.
We say that $G(n,p)$ satisfies a property $\mathcal{P}$ \emph{with high probability} if the probability that $\mathcal{P}$ holds on $G(n,p)$ converges to~$1$ as~$n$ goes to infinity.

First, we show that the twin-width of $G(n,p)$ is proportional to $n\sqrt{p}$ with high probability if $(726\ln n)/n\le p\le1/2$ for all sufficiently large~$n$.
Previously, Ahn, Hendrey, Kim, and Oum~\cite{AHKO2021} showed the following lower bound.

\begin{theorem}[Ahn, Hendrey, Kim, and Oum~\cite{AHKO2021}]\label{thm:ahko}
    For every~$\varepsilon>0$, if $p:=p(n)$ satisfies $1/n\le p\le1/2$, then with high probability,
    \begin{linenomath*}\[ 
        \tww(G(n,p))>2p(1-p)n-(2\sqrt2+\varepsilon)\sqrt{p(1-p)n\ln n}.
    \]\end{linenomath*}
\end{theorem}

They also showed that the twin-width of an $n$-vertex graph is at most $n/2+o(n)$, and so when $1/n\le p\le1/2$, the twin-width of $G(n,p)$ is between $(p-\varepsilon)n$ and $\frac{(1+\varepsilon)}{2}n$ with high probability for any positive constant~$\varepsilon$.
We improve this and determine that $n\sqrt{p}$ is the exact function when $p\ge(726\ln n)/n$.

\begin{theorem}\label{thm:main3}
    Let $p:=p(n)$ such that $\frac{726\ln n}{n}\le p(n)\le1/2$ for all sufficiently large~$n$.
    With high probability, $\tww(G(n,p))=\Theta(n\sqrt{p})$.
\end{theorem}

Ahn, Hendrey, Kim, and Oum~\cite{AHKO2021} computed thresholds for $\tww(G(n,p))\in\{0,1,2\}$.
In particular, they showed that $\tww(G(n,p))\le2$ with high probability if $p\le cn^{-1}$ for any constant $c\in(0,1)$.
Determining the twin-width asymptotically when $\frac{1}{n}\le p<\frac{726\ln n}{n}$ remains open.
In Proposition~\ref{prop:not tiny 1}, we will show that if $\omega(n^{-1})\le p\le1/2$, then the twin-width of $G(n,p)$ is unbounded.

Second, we focus on the situation when~$p$ is a nonzero constant.
Each of Theorems~\ref{thm:ahko} or \ref{thm:main3} tells us that the twin-width of $G(n,p)$ grows linearly in~$n$ for this case, but these theorems do not decide the coefficient of the linear term in $\tww(G(n,p))$.
We show that there is a threshold $p^*\approx 0.4013$ such that if $0<p<p^*$, then there exists $c>0$ depending on $p$ such that with high probability,
\begin{linenomath*}\[
    \tww(G(n,p))>(2p(1-p)+c)n
\]\end{linenomath*}
and if $p^*<p\le1/2$, then with high probability, 
\begin{linenomath*}\[
    \tww(G(n,p))=2p(1-p)n+o(n),
\]\end{linenomath*}
proving that the lower bound given by Theorem~\ref{thm:ahko} is tight up to the first-order term in this case, and furthermore we also determine the second-order term of $\tww(G(n,p))$.
We were told that Behague et al.~\cite{private} independently studied the twin-width of~$G(n,p)$.

\begin{theorem}\label{thm:main}
    There exists $p^*\in(0.4012,0.4013)$ such that the following hold for a real $p$ and $q:=1-p$.
    \begin{enumerate}[label=\rm(\arabic*)]
        \item\label{item:main1} If $p^*<p\le1/2$, then with high probability,
        \begin{linenomath*}\[
            \tww(G(n,p))=2pqn-\sqrt{6pq(1-2pq)n\ln n}+o(\sqrt{n\ln n}).
        \]\end{linenomath*}
        \item\label{item:main2} 
        If $0<p<p^*$, then there exists $c:=c(p)>0$ such that with high probability,
        \begin{linenomath*}\[
            \tww(G(n,p))>(2pq+c)n.
        \]\end{linenomath*}
    \end{enumerate}
\end{theorem}

As a corollary, we deduce the following for $G(n,1/2)$.

\begin{corollary}\label{cor:main}
    The twin-width of $G(n,1/2)$ is $\frac{n}{2}-\frac{\sqrt{3n\ln n}}{2}+o(\sqrt{n\ln n})$ with high probability.
\end{corollary}

One natural question for a graph parameter is to determine its maximum value across all graphs with a fixed number of vertices.
Ahn, Hendrey, Kim, and Oum~\cite{AHKO2021} showed that the twin-width of an $n$-vertex graph is bounded from above by $n/2+\sqrt{n\ln n}/2+o(\sqrt{n\ln n})$
and proved that an $n$-vertex Paley graph has twin-width exactly $(n-1)/2$.
They asked whether there exist $n$-vertex graphs of twin-width larger than $(n-1)/2$.
Corollary~\ref{cor:main} implies that the twin-width of the random graph $G(n,1/2)$ is far less than $(n-1)/2$, differing by at least $\Omega(\sqrt{n\ln n})$.
This behavior of twin-width is different from that of rank-width introduced by Oum and Seymour~\cite{OS2004}.
Like twin-width, rank-width is almost invariant under graph complementation; the rank-width of a graph and that of its complement differ by at most $1$.
By definition, the rank-width of an $n$-vertex graph is bounded from above by $\lceil n/3\rceil$ and Lee, Lee, and Oum~\cite{LLO2012} showed that if $p\in(0,1)$ is a constant, then the rank-width of $G(n,p)$ is $\lceil n/3\rceil-O(1)$ with high probability.
This means that unlike twin-width, the rank-width of $G(n,p)$ is with high probability within an additive constant of the maximum possible value.

We remark that the twin-width is not a monotonic parameter, meaning that the twin-width of a subgraph of a graph~$G$ can be more than the twin-width of~$G$.
It is true however that if~$p_1$ and~$p_2$ are constant with $p^*<p_1<p_2\le1/2$, then $\tww(G(n,p_1))\le\tww(G(n,p_2))$ with high probability by Theorem~\ref{thm:main}.
Moreover, the upper bound arguments to prove Theorem~\ref{thm:main} in Section~\ref{sec:upper} can also be used to show that if~$p_1$ and~$p_2$ are constant with $p_1\le p^*<p_2\le1/2$, then $\tww(G(n,p_1))\le\tww(G(n,p_2))$ with high probability.
In fact, we suspect that this also holds when $p_1<p_2\le p^*$, but this remains open.

\medskip
\noindent\textbf{Organization.}
In Section~\ref{sec:prelim}, we introduce the definitions of a trigraph and twin-width, and summarize some standard terminology and results from graph theory and probability theory.
Section~\ref{sec:upper} is devoted to the proof of the upper bound in Theorem~\ref{thm:main}\ref{item:main1}.
In Section~\ref{sec:lower}, we prove the lower bounds in Theorem~\ref{thm:main} and prove Theorem~\ref{thm:main3}.
The proofs of some technical inequalities are relegated to the appendix.

\section{Preliminaries}\label{sec:prelim}

In this paper, all (tri)graphs are simple and finite.
For a graph~$G$ and a set $X\subseteq V(G)$, let $G[X]$ be the subgraph of~$G$ whose vertex set is~$X$ and whose edge set consists of all edges of~$G$ whose both ends belong to~$X$.
For disjoint sets $A,B\subseteq V(G)$, let $G[A,B]$ be the bipartite subgraph of~$G$ with $V(G[A,B])=A\cup B$ whose edge set consists of all edges in~$G$ incident to one vertex in~$A$ and another vertex in~$B$.
We denote by $\delta(G)$ the minimum degree of~$G$ and by $\Delta(G)$ the maximum degree of~$G$.

We start with a standard lemma that allows us to find a subgraph with large minimum degree of a given graph.

\begin{proposition}[See~{\cite[Proposition~1.2.2]{Diestel2018}}]\label{prop:dense subgraph}
    Let $G$ be a graph with at least one edge.
    Then~$G$ has a subgraph $H$ with minimum degree $\delta(H)>\abs{E(G)}/\abs{V(G)}$.
\end{proposition}

\subsection{Twin-width}

For a positive integer~$t$ and a set~$S$, $\binom{S}{t}$ is the set of all $t$-element subsets of~$S$ and $[t]:=\{1,\ldots,t\}$.
For sets~$S$ and~$T$, $S\triangle T:=(S\setminus T)\cup(T\setminus S)$ and is called the \emph{symmetric difference} of~$S$ and~$T$.

A \emph{trigraph} is a triple $G=(V,B,R)$ for disjoint sets $B,R\subseteq\binom{V}{2}$.
For a trigraph $G=(V,B,R)$, we denote by $V(G):=V$ the set of \emph{vertices}, by $B(G):=B$ the set of \emph{black edges}, by $R(G):=R$ the set of \emph{red edges}, and by $E(G):=B\cup R$ the set of \emph{edges}.
We identify a graph $H=(V,E)$ with a trigraph $H=(V,E,\emptyset)$.
For a set $S\subseteq V(G)$, $G\setminus S$ is the trigraph obtained from~$G$ by deleting the vertices in~$S$ and the edges incident with vertices in~$S$.
If $S=\{v\}$, then we may write $G\setminus v$ for $G\setminus\{v\}$.

For a vertex~$v$ of a trigraph~$G$, the \emph{degree} of~$v$, denoted $\deg(v)$, is the number of edges of~$G$ incident with~$v$, and the \emph{red-degree} of $v$, denoted $\rdeg(v)$, is the number of red edges of~$G$ incident with~$v$.
Let $N_G^R(v)$ be the set of vertices~$w$ such that $vw \in R(G)$ and let $\DeltaR(G)$ be the maximum red-degree of~$G$.

For distinct vertices~$u$ and~$v$ of~$G$, not necessarily adjacent, $G/\{u,v\}=(V',B',R')$ is the trigraph such that for a new vertex $w\notin V$, we have that $V'=(V\setminus\{u,v\})\cup\{w\}$ and $G\setminus\{u,v\}=(G/\{u,v\})\setminus w$, and for each $x\in V'\setminus\{w\}$, the following hold:
\begin{itemize}
    \item $wx\in B'$ if and only if $\{ux,vx\}\subseteq B$,
    \item $wx\notin B'\cup R'$ if and only if $\{ux,vx\}\cap(B\cup R)=\emptyset$, and
    \item $wx\in R'$, otherwise.
\end{itemize}
We say that $G/\{u,v\}$ is obtained from~$G$ by \emph{contracting~$u$ and~$v$}.
Let $r_G(u,v)$ be the red-degree of the new vertex in $G/\{u,v\}$ obtained by contracting~$u$ and~$v$.
Note that if $G$ is a graph, then $r_G(u,v)=\abs{(N_G(u)\triangle N_G(v))\setminus\{u,v\}}$.

A \emph{partial contraction sequence} from~$G_1$ to~$G_t$ is a sequence $G_1,\ldots,G_t$ of trigraphs such that for each $i\in[t-1]$, $G_{i+1}$ is obtained from~$G_i$ by contracting two vertices.
If $\DeltaR(G_i)\le d$ for every $i\in[t]$, then we call the sequence a \emph{partial $d$-contraction sequence}.
A \emph{contraction sequence} of~$G$ is a partial contraction sequence from~$G$ to a $1$-vertex graph.
A \emph{$d$-contraction sequence} of~$G$ is a partial $d$-contraction sequence from~$G$ to a $1$-vertex graph.
The \emph{twin-width} of~$G$ is the minimum~$d$ such that there is a $d$-contraction sequence of~$G$.
We denote by $\tww(G)$ the twin-width of~$G$.

For a partial contraction sequence $\sigma=G_1,G_2,\ldots,G_t$ from~$G_1$ to~$G_t$, we define a function $\sigma:V(G_1)\to V(G_t)$, using the same symbol with the sequence, as follows.
For each $v\in V(G_1)$, $\sigma(v):=v$ if~$v$ is not contracted, and otherwise $\sigma(v)$ is the vertex of~$G_t$ to which~$v$ is contracted.
For a set $S\subseteq V(G_t)$, $\sigma^{-1}(S)$ is the preimage of~$S$ under~$\sigma$.
If $S=\{v\}$, then we may write $\sigma^{-1}(v)$ for $\sigma^{-1}(\{v\})$.
For $s\in[t]$ and a positive integer $i$, we denote by $C^\sigma_{s,i}$ the number of vertices~$v$ of~$G_s$ such that $\abs{\sigma_s^{-1}(v)}=i$, where $\sigma_s=G_1,\ldots,G_s$ is the subsequence consisting of the first $s$ terms of $\sigma$.
We will use the next observation in the proof of Theorem~\ref{thm:main2}.
We may omit the superscript $\sigma$ if it is clear from the context.

\begin{observation}\label{obs}
    Let $\sigma=G_1,G_2,\ldots,G_t$ be a partial contraction sequence from an $n$-vertex trigraph~$G_1$ to~$G_t$.
    For every $s\in[t]$, we have that $\sum_{i\ge1}i\cdot C^\sigma_{s,i}=n$.
\end{observation}

Here is an alternative definition of the twin-width, introduced by Bonnet, Kim, Reinald, and Thomass\'e~\cite{twin-width6}.
For a finite set~$V$, a \emph{partition sequence} of~$V$ is a sequence of partitions $\Pi_1,\ldots,\Pi_{\abs{V}}$ of~$V$ where~$\Pi_1$ is the partition into singletons, and for each $i\in[\abs{V}-1]$, $\Pi_{i+1}$ is obtained from~$\Pi_i$ by merging two parts of~$\Pi_i$.
For a graph $G=(V,E)$ and a partition~$\Pi$ of~$V$, the \emph{red-degree} of a part~$X\in\Pi$ with respect to~$G$ is the number of parts $Y\in\Pi\setminus\{X\}$ such that~$Y$ is neither complete nor anticomplete to $X$ in $G$.
The \emph{width} of a partition~$\Pi$ of~$V$ with respect to~$G$ is the maximum red-degree of parts in~$\Pi$ with respect to~$G$.
The \emph{width} of a partition sequence with respect to~$G$ is the maximum width of partitions in the sequence with respect to~$G$.

For a graph $G=(V,E)$, a partition sequence $\Pi_1,\ldots,\Pi_{\abs{V}}$ of~$V$ \emph{induces} a contraction sequence $G_1,\ldots,G_n$ of~$G$ as follows.
Let $G_1:=G$, and for each $i\in[\abs{V}-1]$, if $\Pi_{i+1}$ is obtained from~$\Pi_i$ by merging two parts~$X$ and~$X'$, then $G_{i+1}$ is obtained from~$G_i$ by contracting the unique two vertices representing~$X$ and~$X'$.
Then it is easy to observe that the width of the partition sequence $\Pi_1,\ldots,\Pi_{\abs{V}}$ with respect to~$G$ is at most~$d$ if and only if the contraction sequence of~$G$ induced by $\Pi_1,\ldots,\Pi_{\abs{V}}$ 
is a $d$-contraction sequence of~$G$. 
Thus, we deduce the following easily. 

\begin{lemma}[Bonnet, Kim, Reinald, and Thomass\'e~\cite{twin-width6}]\label{lem:partionsequence}
    The twin-width of a graph~$G$ equals the minimum width of partition sequences of $V(G)$ with respect to~$G$.
\end{lemma}

We will need the following estimate on the number of partition sequences of $[n]$.

\begin{lemma}\label{lem:number-partition-sequence}
    The number of partition sequences of $[n]$ is at most $\exp(2n\ln n)$.
\end{lemma}
\begin{proof}
    The number of partition sequences of $[n]$ is $\prod_{k=2}^n\binom{k}{2}\le(n^2)^n=\exp(2n\ln n)$.
\end{proof}

\subsection{Probabilistic tools}

We frequently use the following concentration inequalities.

\begin{lemma}[Markov's inequality] \label{markov}
    For a nonnegative random variable $X$ and $t>0$,
    \begin{linenomath*}\[
        \Pr[X\ge t]\le\frac{\E[X]}{t}.
    \]\end{linenomath*}
\end{lemma}

\begin{lemma}[Chebyshev's inequaity] \label{chebyshev}
    For a random variable $X$ and $t>0$,
    \begin{linenomath*}\[
        \Pr[\abs{X-\E[X]}\ge t]\le\frac{\Var(X)}{t^2}.
    \]\end{linenomath*}
\end{lemma}

\begin{lemma}[Chernoff bound; see~{\cite[Chapter 4]{MU2017}}]\label{lem:Chernoffbound}
    Let $X_1,\ldots,X_n$ be mutually independent random variables such that $\Pr[X_i=1]=p_i$ and $\Pr[X_i=0]=1-p_i$ for each $i\in[n]$.
    Let $X:=\sum_{i=1}^nX_i$ and $\mu:=\sum_{i=1}^np_i=\E[X]$.
    Then the following bounds hold.
    \begin{enumerate}[label=(\roman*)]
        \item\label{cond:chernoff1} If $0<\delta\le1$, then
            \begin{linenomath*}\[
                \Pr[X\ge(1+\delta)\mu]\le\exp(-\mu\delta^2/3).
            \]\end{linenomath*}
        \item\label{cond:chernoff2} If $R\ge6\mu$, then 
            \begin{linenomath*}\[
                \Pr[X\ge R]\le2^{-R}.
            \]\end{linenomath*}
        \item\label{cond:chernoff3} If $0<\delta<1$, then 
            \begin{linenomath*}\[
                \Pr[X\le(1-\delta)\mu]\le\exp(-\mu\delta^2/2).
            \]\end{linenomath*}
    \end{enumerate} 
\end{lemma}

We also use the following sharper bound several times, which could be deduced from a more precise version of Chernoff bounds.
We provide a short proof in Appendix~\ref{app:proofs} for the sake of completeness.

\begin{restatable}{lemma}{binomUpper}\label{lem:binomUpper}
    Let $X_1,\ldots,X_n$ be mutually independent random variables such that for each $i\in[n]$, $\Pr[X_i=1]=p_i$ and $\Pr[X_i=0]=1-p_i$.
    Let $X:=\sum_{i=1}^nX_i$ and $p:=\frac{1}{n}\sum_{i=1}^n p_i$.
    If $p\in(0,1)$ and $0<\varepsilon\le3p/10$, then
    \begin{linenomath*}\[
        \Pr[X\le(p-\varepsilon)n]\le\exp\left(-\frac{n\varepsilon^2}{2p(1-p)}+\frac{n\varepsilon^3}{2p^2(1-p)^2}\right).
    \]\end{linenomath*}
\end{restatable}

For a positive integer~$n$ and a real $p\in[0,1]$, 
the \emph{binomial distribution} $\mathcal{B}(n,p)$ is the probability distribution of the random variable $X:=\sum_{i=1}^n X_i$, where $X_1,\ldots,X_n$ are mutually independent Bernoulli random variables such that for each $i\in[n]$, we have $\Pr[X_i=1]=p$ and $\Pr[X_i=0]=1-p$.
The following lemma is an anti-concentration bound for the binomial distribution, whose proof will be provided in Appendix~\ref{app:proofs}.

\begin{restatable}{lemma}{binomLower}\label{lem:binomLower}
    Let $p\in(0,1)$ and $n\ge4$ be an integer.
    Let~$\varepsilon$ be a real such that $\frac{1}{\sqrt{n}}\le\varepsilon\le\min\{\frac{p}{2},1-p\}$.
    Then
    \begin{linenomath*}\[
        \Pr[\mathcal{B}(n,p)\le(p-\varepsilon)n]\ge\frac{1}{2\sqrt{2}}\exp\left(-\frac{n\varepsilon^2}{2p(1-p)}-\frac{3\sqrt{n\varepsilon^2}}{2p(1-p)}-\frac{4n\varepsilon^3}{p^2(1-p)^2}\right).
    \]\end{linenomath*}
\end{restatable}

To estimate error probabilities, we often use the following three lemmas.
The first one is a standard bound on the binomial coefficient.
The second one contains two elementary inequalities, and the third one is a simple corollary of the second lemma.

\begin{lemma}[See~{\cite[Theorem~3.6.1]{MN2009}}]\label{lem:Stirling}
    For integers $1\le k\le n$, we have $\sum_{i=0}^k\binom{n}{i}\le(\frac{en}{k})^k$.
\end{lemma}

\begin{lemma}\label{lem:pq}
    Let $p\in[0,1]$, $q:=1-p$, and $k\ge2$ be an integer.
    \begin{enumerate}[label=\rm(\roman*)]
        \item \label{item:pq-1} $\frac{1-p^k-q^k}{k}\le\frac{1-p^{k-1}-q^{k-1}}{k-1}-\frac{2^k-2k-2}{k(k-1)}(pq)^{k/2}$.
        \item \label{item:pq-2} $1-p^k-q^k\le kpq-((k-4)2^{k-2}+2)(pq)^{k/2}$.
    \end{enumerate}
\end{lemma}
Note that from \ref{item:pq-2}, we deduce that $1-p^k-q^k\le kpq$ for all positive integers~$k$ because $1-p-q=0\le pq$ and $(k-4)2^{k-2}+2\ge0$ for all $k\ge2$.

\begin{proof}
    Observe that $p^k=p^{k-1}-pqp^{k-2}$, $q^k=q^{k-1}-pqq^{k-2}$, and therefore
    \begin{linenomath*}\begin{equation}\label{eq:pq-recurrent}
        1-p^k-q^k=1-p^{k-1}-q^{k-1}+pq(p^{k-2}+q^{k-2}).
    \end{equation}\end{linenomath*}
    
    We first show \ref{item:pq-1}.
    Since $1-p^2-q^2=2pq$ and $1-p^3-q^3=3pq$, \ref{item:pq-1} holds for $k\in\{2,3\}$.
    Thus, we may assume that $k\ge4$.
    We have that
    \begin{linenomath*}\begin{align*}
        &pq(p^{k-2}+q^{k-2})\\
        &=\frac1{k}\left((p+q)^{k}-p^k-q^k-\sum_{i=2}^{k-2}\binom{k}{i}\frac{p^iq^{k-i}+p^{k-i}q^i}{2}\right)\tag*{by the binomial theorem}\\
        &\le\frac{1}{k}\left((1-p^k-q^k)-\sum_{i=2}^{k-2}\binom{k}{i}p^{k/2}q^{k/2}\right)\tag*{by the AM--GM inequality}\\
        &=\frac{1}{k}\left((1-p^k-q^k)-(2^k-2k-2)(pq)^{k/2}\right)\tag*{by the binomial theorem.}
    \end{align*}\end{linenomath*}
    By \eqref{eq:pq-recurrent}, it follows that 
    \begin{linenomath*}\[ 
        \frac{k-1}{k} (1-p^k-q^k)
        \le1-p^{k-1}-q^{k-1} -\frac{1}{k} (2^k-2k-2)(pq)^{k/2}, 
    \]\end{linenomath*}
    which implies \ref{item:pq-1}.
    
    To see \ref{item:pq-2}, we first claim the following inequality:
    \begin{linenomath*}\begin{equation}\label{eq:pq-recurrent2}
        1-p^k-q^k\le1-p^{k-1}-q^{k-1}+pq-(2^{k-2}-2)(pq)^{k/2}.
    \end{equation}\end{linenomath*}
    It is easy to see that the claim holds for $k\in\{2,3\}$, so we may assume that $k\ge4$.
    Then,
    \begin{linenomath*}\begin{align*}
        &p^{k-2}+q^{k-2}\\
        &=(p+q)^{k-2}-\sum_{i=1}^{k-3}\binom{k-2}{i}\frac{p^iq^{k-2-i}+p^{k-2-i}q^i}{2}\tag*{by the binomial theorem}\\
        &\le1-\sum_{i=1}^{k-3}\binom{k-2}{i}(pq)^{k/2-1}\tag*{by the AM--GM inequality}\\
        &=1-(2^{k-2}-2)(pq)^{k/2-1}\tag*{by the binomial theorem.}
    \end{align*}\end{linenomath*}
    Combined with~\eqref{eq:pq-recurrent}, we deduce the claim.

    We now show \ref{item:pq-2} by induction on~$k$.
    Observe that \ref{item:pq-2} is trivially true if $k=2$.
    Thus we may assume that $k\ge3$.
    By~\eqref{eq:pq-recurrent2} and the inductive hypothesis,
    \begin{linenomath*}\begin{equation}
        \label{eq:pq-recurrent3}
        1-p^k-q^k\le kpq-((k-5)2^{k-3}+2)(pq)^{(k-1)/2}-(2^{k-2}-2)(pq)^{k/2}.
    \end{equation}\end{linenomath*}
    Since $\sqrt{pq}\le1/2$, we have that
    \begin{linenomath*}\[
        ((k-5)2^{k-3}+2)(pq)^{(k-1)/2}=((k-5)2^{k-2}+4)\frac{(pq)^{(k-1)/2}}{2}\ge((k-5)2^{k-2}+4)(pq)^{k/2}.
    \]\end{linenomath*}
    Combined with~\eqref{eq:pq-recurrent3}, we derive \ref{item:pq-2}.
    This completes the proof.
\end{proof}

\begin{lemma}\label{lem:pqadd}
    Let $p\in(0,1)$, $q:=1-p$, and $m,n\ge2$ be integers.
    Then
    \begin{linenomath*}\[
        (1-p^m-q^m)+(1-p^n-q^n)>1-p^{m+n}-q^{m+n}.
    \]\end{linenomath*}
\end{lemma}
\begin{proof}
    The lemma follows from adding the following inequalities that hold by Lemma~\ref{lem:pq}\ref{item:pq-1}, 
    \begin{linenomath*}\[
        1-p^n-q^n>\frac{n}{m+n}(1-p^{m+n}-q^{m+n}),
    \]\end{linenomath*}
    \begin{linenomath*}\[
        1-p^m-q^m>\frac{m}{m+n}(1-p^{m+n}-q^{m+n}).\qedhere
    \]\end{linenomath*}
\end{proof}

We say that the event~$A_n$ occurs \emph{with high probability} if the probability of~$A_n$ tends to~$1$ as $n\to\infty$ (in notation, $\Pr[A_n]=1-o(1))$.
When $\Pr[A_n]=1-\exp(-\omega(\ln n))$, we often say that~$A_n$ occurs \emph{with very high probability}.
It is sometimes convenient to use this notion while applying union bound.
In particular, if there are $O(n^c)$ events that hold with very high probability for some constant $c>0$, then their union also holds with very high probability.
Throughout this paper, we implicitly assume that the graphs we consider have sufficiently many vertices to support our arguments whenever necessary.

\section{Upper bounds for the twin-width of random graphs}\label{sec:upper}

We prove the upper bound of Theorem~\ref{thm:main}\ref{item:main1}.
We start with defining the precise value of~$p^*$.
Let $\alpha,\beta:(0,1)\to\mathbb{R}$ be functions such that
\begin{linenomath*}\begin{align}
    \alpha(x)&:=\frac{x(1-x)}{2(1-x^3-(1-x)^3)-(1-x^6-(1-x)^6)}\label{eq:alpha},\\
    \beta(x)&:=\frac{2x(1-x)+(1-x^8-(1-x)^8)-(1-x^{12}-(1-x)^{12})}{3(1-x^8-(1-x)^8)-2(1-x^{12}-(1-x)^{12})},\label{eq:beta}
\end{align}\end{linenomath*}
which are well-defined by Lemma~\ref{lem:pq}\ref{item:pq-1}.
We will show in Lemma~\ref{lem:ab} that there is a unique $p^*\in(0,1/2]$ such that $\alpha(p^*)=\beta(p^*)$, and that $0.4012<p^*<0.4013$.

\begin{theorem}\label{thm:upper2}
    If $p^*<p\le1/2$ and $q:=1-p$, then with high probability,
    \begin{linenomath*}\[
        \tww(G(n,p))\le2pqn-\sqrt{6pq(1-2pq)n\ln n}+o(\sqrt{n \ln n}).
    \]\end{linenomath*}
\end{theorem}

\medskip
\noindent\textbf{Overview.}
Prior to a technical proof of Theorem~\ref{thm:upper2}, we provide an overview of the proof with an intuition about why $\tww(G(n,p))$ has different behaviors according to~$p$ compared with~$p^*$.
We will prove Theorem~\ref{thm:upper2} by presenting a specific contraction sequence.
In $G(n,p)$, when we first execute a contraction, the expected red-degree of the new vertex is exactly $2pq(n-2)$, which is already larger than the upper bound in Theorem~\ref{thm:upper2}.
Thus, the initial contractions we perform must be selected carefully.

Our approach will be to consider a partition $(A,B)$ of $V(G)$ with $\abs{A}=\lfloor n^{1-\delta}\rfloor$ for some small positive constant $\delta$ (any $\delta\in(0,0.005)$ suffices).
Let us first present a prototype of the contraction sequence as follows.
For $m:=\abs{B}$, let $a:=\lfloor\alpha(p)\cdot n\rfloor$.
For simplicity, we assume that both~$m$ and~$a$ are even.
We will later show in Lemma~\ref{lem:ab} that $1/3<\alpha(0.5)\le\alpha(p)\le\alpha(0.4)<1/2$ if $p^*<p\le1/2$.

\begin{enumerate}[label=\bf{Step \arabic*.},leftmargin=*]
    \item For some small positive constant $\varepsilon$ (any $\varepsilon\in(0,0.005)$ suffices), carefully select $\lfloor n^{1/2+\varepsilon}\rfloor$ disjoint pairs of vertices in~$A$ and contract each of the pairs.
    \item Contract  $a$ disjoint pairs of vertices in~$B$ and denote the new vertices by $x_1,\ldots,x_a$.
    \item For each $i\in[m-2a]$, contract $x_i$ and a non-contracted vertex to create a new vertex~$y_i$.
    \item For each $j\in[(3a-m)/2]$, contract $x_{m-2a+(2j-1)}$ and $x_{m-2a+2j}$ to create a new vertex~$z_j$.
    \item Take an arbitrary contraction sequence of the resulting trigraph.
\end{enumerate}

In Step 1, we remark that if we contract two vertices~$u$ and~$v$ in~$A$, the red-degree of the new vertex in $G[A,B]/\{u,v\}$ has a binomial distribution $\mathcal{B}(m,2pq)$.
Thus, if we randomly contract two vertices in~$A$, the red-degree of the new vertex is close to a normal distribution with mean $2pqm$ and variance $2pq(1-2pq)m$ for sufficiently large $m$.
Then via Lemmas~\ref{lem:binomUpper} and~\ref{lem:binomLower}, we can show that with high probability, there are at least $n^{1/2+\varepsilon}$ disjoint pairs of vertices in~$A$ such that for each of the pairs, if we contract it, then the red-degree of the new vertex is bounded from above by the upper bound in~Theorem~\ref{thm:upper2}; see Lemma~\ref{lem:pairs general}.
Those $n^{1/2+\varepsilon}$ pairs will be our choices in Step~1.
In the proof of Lemma~\ref{lem:pairs general}, we only use the information from $G[A,B]$ and preserve the randomness in~$G[B]$ for Steps~2--4.

After Step~1, the number of vertices has decreased sufficiently so that the expected number of new red edges created by contracting a pair of vertices in~$B$ is not too high.
This allows us to contract vertices in~$B$ according to Steps~2--4.

In Step~4, we can show that expected red-degree of a new vertex created in this step is bounded from above by $((1-p^8-q^8)(3a-1)+(1-p^{12}-q^{12})(1-2a))n$, which is much larger than $2pqn$ if $0<p<p^*$, and is smaller than $(2pq-c')n$ for some constant $c'>0$ if $p^*<p\le1/2$.
This is the crucial reason that the threshold $p^*$ exists, as we will further discuss in Section~\ref{subsec:less than p star}.

The reason why we call this fixed contraction sequence a prototype is that following this contraction sequence, some of the contractions in Steps 2--4 may create a vertex whose red-degree is higher than the desired upper bound.
However, we can show that with high probability, this only occurs $O(\sqrt{n})$ times; see Lemma~\ref{lem:exposingBA}.
This means that we can skip these \emph{bad} contractions, while still performing sufficiently many contractions to reduce the total number of vertices down to the intended upper bound on the twin-width.
The remainder of the contraction sequence can then be completed arbitrarily as in Step~5.

In Subsection~\ref{subsec:nonrandom}, we present Lemma~\ref{lem:nonrandom}, a technical yet straightforward lemma that gives an upper bound on the twin-width of an arbitrary graph satisfying a list of properties with respect to a set of parameters.
In Subsection~\ref{subsec:random}, we choose values of the parameters in Lemma~\ref{lem:nonrandom}, which allow us to prove Theorem~\ref{thm:upper2}.
The bulk of the proof is broken up into smaller lemmas, in which we show that with high probability, the random graph $G(n,p)$ with $p\in(p^*,1/2]$ satisfies each of the properties of the technical lemma with respect to these values.
In Subsection~\ref{subsec:upper}, we combine all these results to finish the proof.

\subsection{Finding a contraction sequence}\label{subsec:nonrandom}

For a graph $G$ and a set $\Pi$ of disjoint subsets of $V(G)$, let $G/\Pi$ be a trigraph with $V(G/\Pi)=\Pi \cup \{\{v\} : v \in V(G)\setminus \bigcup_{S\in\Pi} S \}$ such that for distinct $S$, $T\in V(G/\Pi)$,
\begin{enumerate}[label=\rm(\roman*)]
    \item $S$ and $T$ are joined by a black edge in $G/\Pi$ if $vw \in E(G)$ for all $v\in S$ and $w\in T$,
    \item $S$ and $T$ are not adjacent in $G/\Pi$ if $vw \not\in E(G)$ for all $v\in S$ and $w\in T$, and
    \item $S$ and $T$ are joined by a red edge otherwise.
\end{enumerate}
We will frequently use the observation that for distinct $S,T \in \Pi$, we have
\begin{linenomath*}\[
    \rdeg_{G/(\Pi\setminus\{S\})}(T)\le\rdeg_{G/\Pi}(T)+\abs{S}-1.
\]\end{linenomath*}

We now define a sequence of partitions corresponding to the (partial) contraction sequence handled in the overview.
Let $m$ and $a$ be positive integers such that $2a\le m\le3a$, and let $\mathcal{B}^{m,a}_0$ be a partition of~$[m]$ into the singletons, that is, $\mathcal{B}^{m,a}_0:=\{\{j\}:j\in[m]\}$.
For each integer $1\le i \le(m+a)/2$, we define a coarser partition $\mathcal{B}^{m,a}_i$ of~$[m]$ from $\mathcal{B}_{i-1}^{m,a}$ by replacing parts
\begin{itemize}
    \item $\{2i-1\}$ and $\{2i\}$ with $x_i:=\{2i-1,2i\}$ if $i\le a$,
    \item $x_{i-a}$ and $\{2a+(i-a)\}$ with their union $y_{i-a}$ if $a<i\le m-a$, and
    \item $x_{(m-2a)+2(i-(m-a))-1}$ and $x_{(m-2a)+2(i-(m-a))}$ with their union $z_{i-(m-a)}$ if $i>m-a$.
\end{itemize}
See figure~\ref{fig:partitions} for an illustration.

\begin{figure}
    \centering
    \begin{tikzpicture}
        \draw (-0.4,-0.9) rectangle (3.4,4.5);
        
        \draw[rounded corners=1mm, fill=gray!20] (0.5,3.8) rectangle (1.5,4.0);
        
        \node[shape=circle,fill=black, scale=0.25] () at (0.6,3.9) {};
        \node[shape=circle,fill=black, scale=0.25] () at (1.4,3.9) {};
        \node[shape=circle,fill=black, scale=0.25] () at (2.4,3.9) {};
        \node () at (0.6,3.675) {\tiny{$1$}};
        \node () at (1.4,3.675) {\tiny{$2$}};
        \node () at (2.4,3.7) {\tiny{$2a+1$}};

        \node () at (0.35,3.9) {\tiny{$x_1$}};
        
        \node[shape=circle,fill=black, scale=0.25] () at (0.6,3.3) {};
        \node[shape=circle,fill=black, scale=0.25] () at (1.4,3.3) {};
        \node[shape=circle,fill=black, scale=0.25] () at (2.4,3.3) {};
        \node () at (0.6,3.1) {\tiny{$3$}};
        \node () at (1.4,3.1) {\tiny{$4$}};
        \node () at (2.4,3.1) {\tiny{$2a+2$}};
        
        \node () at (1.0,2.7) {\tiny{$\vdots$}};
        \node () at (2.4,2.7) {\tiny{$\vdots$}};
        
        \node[shape=circle,fill=black, scale=0.25] () at (0.6,2.1) {};
        \node[shape=circle,fill=black, scale=0.25] () at (1.4,2.1) {};
        \node[shape=circle,fill=black, scale=0.25] () at (2.4,2.1) {};
        \node () at (0.6,1.9) {\tiny{$2b-1$}};
        \node () at (1.4,1.9) {\tiny{$2b$}};
        \node () at (2.4,1.9) {\tiny{$m$}};
        
        \node[shape=circle,fill=black, scale=0.25] () at (0.6,1.5) {};
        \node[shape=circle,fill=black, scale=0.25] () at (1.4,1.5) {};
        \node () at (0.6,1.3) {\tiny{$2b+1$}};
        \node () at (1.4,1.3) {\tiny{$2b+2$}};
        
        \node[shape=circle,fill=black, scale=0.25] () at (0.6,0.9) {};
        \node[shape=circle,fill=black, scale=0.25] () at (1.4,0.9) {};
        \node () at (0.6,0.7) {\tiny{$2b+3$}};
        \node () at (1.4,0.7) {\tiny{$2b+4$}};
        
        \node () at (1.0,0.3) {\tiny{$\vdots$}};
        
        \node[shape=circle,fill=black, scale=0.25] () at (0.6,-0.3) {};
        \node[shape=circle,fill=black, scale=0.25] () at (1.4,-0.3) {};
        \node () at (0.6,-0.5) {\tiny{$2a-1$}};
        \node () at (1.4,-0.5) {\tiny{$2a$}};
        
        \node () at (1.5,-1.25) {$\mathcal{B}_1$};
    \end{tikzpicture}
    \hspace{0.5cm}
    \begin{tikzpicture}
        \draw (-0.4,-0.9) rectangle (3.4,4.5);
        
        \draw[rounded corners=1mm, fill=gray!20] (0.5,3.8) rectangle (2.5,4.0);
        \node[shape=circle,fill=black, scale=0.25] () at (0.6,3.9) {};
        \node[shape=circle,fill=black, scale=0.25] () at (1.4,3.9) {};
        \node[shape=circle,fill=black, scale=0.25] () at (2.4,3.9) {};
        \node () at (0.6,3.675) {\tiny{$1$}};
        \node () at (1.4,3.675) {\tiny{$2$}};
        \node () at (2.4,3.675) {\tiny{$2a+1$}};
        
        \draw[rounded corners=1mm, fill=gray!20] (0.5,3.2) rectangle (1.5,3.4);
        \node[shape=circle,fill=black, scale=0.25] () at (0.6,3.3) {};
        \node[shape=circle,fill=black, scale=0.25] () at (1.4,3.3) {};
        \node[shape=circle,fill=black, scale=0.25] () at (2.4,3.3) {};
        \node () at (0.6,3.075) {\tiny{$3$}};
        \node () at (1.4,3.075) {\tiny{$4$}};
        \node () at (2.4,3.1) {\tiny{$2a+2$}};
        
        \node () at (1.0,2.7) {\tiny{$\vdots$}};
        \node () at (2.4,2.7) {\tiny{$\vdots$}};
        
        \draw[rounded corners=1mm, fill=gray!20] (0.5,2.0) rectangle (1.5,2.2);
        \node[shape=circle,fill=black, scale=0.25] () at (0.6,2.1) {};
        \node[shape=circle,fill=black, scale=0.25] () at (1.4,2.1) {};
        \node[shape=circle,fill=black, scale=0.25] () at (2.4,2.1) {};
        \node () at (0.6,1.875) {\tiny{$2b-1$}};
        \node () at (1.4,1.875) {\tiny{$2b$}};
        \node () at (2.4,1.9) {\tiny{$m$}};
        
        \draw[rounded corners=1mm, fill=gray!20] (0.5,1.4) rectangle (1.5,1.6);
        \node[shape=circle,fill=black, scale=0.25] () at (0.6,1.5) {};
        \node[shape=circle,fill=black, scale=0.25] () at (1.4,1.5) {};
        \node () at (0.6,1.275) {\tiny{$2b+1$}};
        \node () at (1.4,1.275) {\tiny{$2b+2$}};
        
        \draw[rounded corners=1mm, fill=gray!20] (0.5,0.8) rectangle (1.5,1.0);
        \node[shape=circle,fill=black, scale=0.25] () at (0.6,0.9) {};
        \node[shape=circle,fill=black, scale=0.25] () at (1.4,0.9) {};
        \node () at (0.6,0.675) {\tiny{$2b+3$}};
        \node () at (1.4,0.675) {\tiny{$2b+4$}};
        
        \node () at (1.0,0.3) {\tiny{$\vdots$}};
        
        \draw[rounded corners=1mm, fill=gray!20] (0.5,-0.4) rectangle (1.5,-0.2);
        \node[shape=circle,fill=black, scale=0.25] () at (0.6,-0.3) {};
        \node[shape=circle,fill=black, scale=0.25] () at (1.4,-0.3) {};
        \node () at (0.6,-0.525) {\tiny{$2a-1$}};
        \node () at (1.4,-0.525) {\tiny{$2a$}};

        \node () at (0.35,3.9) {\tiny{$y_{1}$}};
        \node () at (0.35,3.3) {\tiny{$x_2$}};
        \node () at (0.35,2.1) {\tiny{$x_b$}};
        \node () at (0.20,1.5) {\tiny{$x_{b+1}$}};
        \node () at (0.20,0.9) {\tiny{$x_{b+2}$}};
        \node () at (0.35,-0.3) {\tiny{$x_a$}};
        
        \node () at (1.5,-1.25) {$\mathcal{B}_{a+1}$};
    \end{tikzpicture}
    \hspace{0.5cm}
    \begin{tikzpicture}
        \draw (-0.4,-0.9) rectangle (3.4,4.5);
        
        \draw[rounded corners=1mm, fill=gray!20] (0.5,3.8) rectangle (2.5,4.0);
        \node[shape=circle,fill=black, scale=0.25] () at (0.6,3.9) {};
        \node[shape=circle,fill=black, scale=0.25] () at (1.4,3.9) {};
        \node[shape=circle,fill=black, scale=0.25] () at (2.4,3.9) {};
        \node () at (0.6,3.675) {\tiny{$1$}};
        \node () at (1.4,3.675) {\tiny{$2$}};
        \node () at (2.4,3.675) {\tiny{$2a+1$}};
        
        \draw[rounded corners=1mm, fill=gray!20] (0.5,3.2) rectangle (2.5,3.4);
        \node[shape=circle,fill=black, scale=0.25] () at (0.6,3.3) {};
        \node[shape=circle,fill=black, scale=0.25] () at (1.4,3.3) {};
        \node[shape=circle,fill=black, scale=0.25] () at (2.4,3.3) {};
        \node () at (0.6,3.075) {\tiny{$3$}};
        \node () at (1.4,3.075) {\tiny{$4$}};
        \node () at (2.4,3.075) {\tiny{$2a+2$}};
        
        \node () at (1.0,2.7) {\tiny{$\vdots$}};
        \node () at (2.4,2.7) {\tiny{$\vdots$}};
        
        \draw[rounded corners=1mm, fill=gray!20] (0.5,2.0) rectangle (2.5,2.2);
        \node[shape=circle,fill=black, scale=0.25] () at (0.6,2.1) {};
        \node[shape=circle,fill=black, scale=0.25] () at (1.4,2.1) {};
        \node[shape=circle,fill=black, scale=0.25] () at (2.4,2.1) {};
        \node () at (0.6,1.875) {\tiny{$2b-1$}};
        \node () at (1.4,1.875) {\tiny{$2b$}};
        \node () at (2.4,1.875) {\tiny{$m$}};
        
        \draw[rounded corners=1mm, fill=gray!20] (0.5,0.8) rectangle (1.5,1.6);
        \node[shape=circle,fill=black, scale=0.25] () at (0.6,1.5) {};
        \node[shape=circle,fill=black, scale=0.25] () at (1.4,1.5) {};
        
        \node[shape=circle,fill=black, scale=0.25] () at (0.6,0.9) {};
        \node[shape=circle,fill=black, scale=0.25] () at (1.4,0.9) {};
        
        \node () at (1.9,1.475) {\tiny{$2b+1$}};
        \node () at (1.9,1.275) {\tiny{$2b+2$}};
        \node () at (1.9,1.075) {\tiny{$2b+3$}};
        \node () at (1.9,0.875) {\tiny{$2b+4$}};
        
        \node () at (1.0,0.3) {\tiny{$\vdots$}};
        
        \draw[rounded corners=1mm, fill=gray!20] (0.5,-0.4) rectangle (1.5,-0.2);
        \node[shape=circle,fill=black, scale=0.25] () at (0.6,-0.3) {};
        \node[shape=circle,fill=black, scale=0.25] () at (1.4,-0.3) {};
        \node () at (0.6,-0.525) {\tiny{$2a-1$}};
        \node () at (1.4,-0.525) {\tiny{$2a$}};

        \node () at (0.35,3.9) {\tiny{$y_{1}$}};
        \node () at (0.35,3.3) {\tiny{$y_{2}$}};
        \node () at (0.10,2.1) {\tiny{$y_{m-2a}$}};
        \node () at (0.35,1.2) {\tiny{$z_{1}$}};
        \node () at (0.35,-0.3) {\tiny{$x_a$}};
        
        \node () at (1.5,-1.25) {$\mathcal{B}_{a+b+1}$};
    \end{tikzpicture}
    \caption{Description of partitions $\mathcal{B}^{n,a}_i$ of $[n]$, where $i \in \{1, a+1, a+b+1\}$ and $b := n-2a$.
    Each gray region represents a part of size at least two in $\mathcal{B}_i$.
    A part $y_1$ of size $3$ appears first in $\mathcal{B}_{a+1}$, and a part $z_1$ of size $4$ appears first in $\mathcal{B}_{a+b+1}$.
    }
    \label{fig:partitions}
\end{figure}
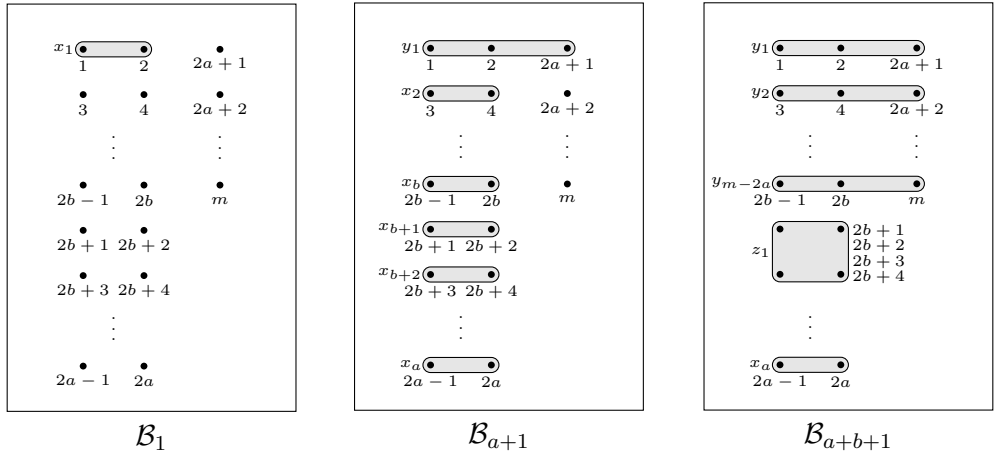

\begin{lemma}\label{lem:nonrandom}
    Let $\alpha\in(\frac{1}{3},\frac{1}{2})$, let $n$, $m$, $s$, and $\ell$ be positive integers with $2\alpha n\le m\le n-2s$, and let $G$ be a graph with $V(G)=[n]$.
    Let $a:=\lfloor\alpha n\rfloor$, $B:=[m]$, and $A:=[n]\setminus[m]$.
    
    Let $\{u_1,v_1\},\ldots,\{u_s,v_s\}$ be disjoint pairs of vertices in $A$ and let $\mathcal{A}_0:=\{\{j\}:j\in A\}$.
    For $1\le i\le s$, we define a coarser partition $\mathcal{A}_i:=(\mathcal{A}_{i-1}\setminus\{\{u_i\},\{v_i\}\})\cup\{u_i,v_i\}$ of $A$.
    For $0\le i \le(m+a)/2$, let $\mathcal{B}_i := \mathcal{B}^{m,a}_{i}$, and let~$L$ be a subset of~$B$ with $\abs{L}\le\ell$.
    
    Let $\lambda_1$, $\lambda_2$, $\lambda_2'$, $\lambda_3$, $\mu_1$, $\mu_2$, $\rho_2$, $\rho_3$, $\nu_2$, $\nu_3$, and~$\nu_4$ be nonnegative reals.
    Suppose that~$G$ satisfies the following properties.
    \begin{enumerate}[label=\rm(\Alph*)]
        \item\label{item:expABpairs} $\rdeg_{G[A,B]/\mathcal{A}_s}(\{u_i,v_i\})\le\lambda_1$ for every $i\in[s]$.
        \item\label{item:expAB} For every $S \subseteq A$ with $\abs{S} \in \{1,2\}$ and every positive integer $i \le(m+a)/2$,
            \begin{linenomath*}\[
                \rdeg_{G[A,B]/(\{S\}\cup\mathcal{B}_i)}(S)\le\mu_{\abs{S}}.
            \]\end{linenomath*}
        \item\label{item:expA} For each $i,j \in[s]$ with $j\le i$,
            \begin{linenomath*}\[
                \rdeg_{G[A]/\mathcal{A}_{i}}(\{u_j,v_j\})
                \le\lambda_2
                \quad \text{and} \quad
                \rdeg_{G[A]/\mathcal{A}_{s}}(\{u_j,v_j\})
                \le\lambda_2'.
            \]\end{linenomath*}
        \item\label{item:expBA} For each positive integer $i\le(m+a)/2$ and $S \in \mathcal{B}_{i}$ with $\abs{S} \in \{2,3\}$ and $S \cap L=\emptyset$,
            \begin{linenomath*}\[ 
                \rdeg_{G[A,B]/(\mathcal{A}_s \cup \mathcal{B}_i)}(S)\le\rho_{\abs{S}}.
            \]\end{linenomath*}
        \item\label{item:expBsingles} For each positive integer $i \le(m+a)/2$ and $u\in B$,
            \begin{linenomath*}\[
                \rdeg_{G[B] / \{S \in \mathcal{B}_i : u \notin S \}} (\{u\})\le\lambda_3
            \]\end{linenomath*}
        \item\label{item:expB} For each positive integer $i \le(m+a)/2$ and $S \in \mathcal{B}_i$ 
            \begin{linenomath*}\[
                \rdeg_{G[B]/\mathcal{B}_i}(S)\le\nu_{\abs{S}}.
            \]\end{linenomath*}
    \end{enumerate}
    Then the twin-width of $G$ is bounded from above by the maximum of the following eight numbers.
    \begin{linenomath*}\[
        \begin{array}{llll}
            n-s-\lfloor (m+a)/2 \rfloor+3\ell,
            &s+\mu_1,
            &\lambda_1+\lambda_2,
            &\lambda_2'+\mu_2+3\ell, \\
            s+\lambda_3,
            &\rho_2+\nu_2+3\ell,
            &\rho_3+\nu_3+3\ell,
            &n-m-s+\nu_4+3\ell.
        \end{array}
    \]\end{linenomath*}
\end{lemma}
\begin{proof}
    Let $r := \lfloor (m+a) / 2 \rfloor$.
    We define a sequence of partitions $\mathcal{B}_0^{\text{good}},\ldots,\mathcal{B}_r^{\text{good}}$ of~$B$ similar to $\mathcal{B}_0,\ldots,\mathcal{B}_r (= \mathcal{B}_i^{m,a},\ldots, \mathcal{B}_r^{m,a})$ where $\mathcal{B}_0^{\text{good}}=\mathcal{B}_0$, but the difference is that we skip to merge two parts $S_1$ and $S_2$ in $\mathcal{B}_{i-1}^{\text{good}}$ if there is $T\in \mathcal{B}_r$ such that $S_1 \cup S_2 \subseteq T$ and $T\cap L \ne \emptyset$.
    More precisely, for each $i\in[r]$,
    \begin{linenomath*}\begin{align*}
        \mathcal{B}_i^{\text{good}}
        :=&\{S \in \mathcal{B}_i:\nexists T\in\mathcal{B}_r\text{ such that }S\subseteq T\text{ and }T\cap L\ne\emptyset\}\\
        &\hspace{0.1cm}\cup\{\{v\}:\exists T\in\mathcal{B}_r\text{ such that }v\in T\text{ and }T\cap L\ne\emptyset\}.
    \end{align*}\end{linenomath*}
    
    For $1\le i\le s+r$, we define partitions $\Pi_i$ and $\Pi_i'$ of $V(G)$ as follows:
    \begin{linenomath*}\begin{align*}
        \Pi_i:=
        \begin{cases}
            \mathcal{A}_i\cup\mathcal{B}_0 & \text{if $i\le s$,}\\
            \mathcal{A}_s\cup\mathcal{B}_{i-s}^{\text{good}} & \text{otherwise}
        \end{cases}
        \hspace{0.5cm}
        \text{and}
        \hspace{0.5cm}
        \Pi_i':=
        \begin{cases}
            \mathcal{A}_i\cup\mathcal{B}_0 & \text{if $i\le s$,}\\
            \mathcal{A}_s\cup\mathcal{B}_{i-s} & \text{otherwise}.
        \end{cases}
    \end{align*}\end{linenomath*}
    Observe that for each $i\in[s+r]$, $\Pi_i'$ is a refinement of $\Pi_{i+1}'$ with $\abs{\Pi_{i}'}=\abs{\Pi_{i+1}'}+1$.
    In addition, $\Pi_i$ is a refinement of~$\Pi_{i+1}$ and $\abs{\Pi_{i}}\le\abs{\Pi_{i+1}}+1$.
    
    For each $i\in[s+r]$, let $G_1:=G$ and
    \begin{linenomath*}\[
        G_{n-\abs{\Pi_i}+1}:=G/\Pi_i.
    \]\end{linenomath*}
    Then $G_1,G_2,\ldots,G_{n-\abs{\Pi_{s+r}}+1}$ is a partial contraction sequence of~$G$.
    We use this partial contraction sequence to obtain the required upper bound for the twin-width of~$G$.
    
    Observe that $\Pi_i$ is a refinement of $\Pi_{i}'$ determined by $L$ and each element in $\Pi_i'$ has size at most~$4$, which implies that $\abs{\Pi_{i}}\le\abs{\Pi_i'}+3\ell$.
    Then
    \begin{linenomath*}\begin{align*}
        |V(G_{n-|\Pi_{s+t}|+1})|=|\Pi_{s+r}|
        \le|\Pi'_{s+r}|+3\ell 
        &=n-(s+r)+3\ell \\
        &=n-s-\lfloor(m+a) / 2 \rfloor+3\ell.
    \end{align*}\end{linenomath*}
    Therefore, it suffices to show that for each $i\in[s+r]$ and $S \in \Pi_i$, the red-degree of~$S$ in~$G/\Pi_i$ is bounded from above by the maximum of
    \begin{linenomath*}\[
        \begin{array}{llll}
            s+\mu_1,
            &\lambda_1+\lambda_2,
            &\lambda_2'+\mu_2+3\ell, \\
            s+\lambda_3,
            &\rho_2+\nu_2+3\ell,
            &\rho_3+\nu_3+3\ell,
            &n-m-s+\nu_4+3\ell.
        \end{array}
    \]\end{linenomath*}

    \medskip
    \noindent\textbf{Case I.} $\abs{S}=1$.
    
    If $i \le s$, then $\rdeg_{G/\Pi_i}(S)\le s \le s+\mu_1$.
    If $i>s$ and $S\subseteq A$, then by~\ref{item:expAB},
    \begin{linenomath*}\[
        \rdeg_{G/\Pi_i}(S)
        \le
        \rdeg_{G/\Pi_i'}(S)
        =
        \rdeg_{G[A] / \mathcal{A}_s}(S)+\rdeg_{G[A,B] / \mathcal{B}_{i-s}}(S)
        \le s+\mu_1.
    \]\end{linenomath*}
    If $i>s$ and $S\subseteq B$, then by~\ref{item:expBsingles},
    \begin{linenomath*}\begin{align*}
        \rdeg_{G/\Pi_i}(S)
        &=
        \rdeg_{G/\mathcal{A}_s}(S)+\rdeg_{G[B]/\mathcal{B}_{i-s}^{\text{good}}}(S)\\
        &\le
        s+\rdeg_{G[B]/\{T \in \mathcal{B}_{i-s} : S \not\subseteq T\}} (S)
        \le
        s+\lambda_3.
    \end{align*}\end{linenomath*}

    \medskip
    \noindent\textbf{Case II.} $\abs{S}=2$.
    
    If $i \le s$, then $S\subseteq A$ and by~\ref{item:expA} and~\ref{item:expABpairs},
    \begin{linenomath*}\begin{align*}
        \rdeg_{G/\Pi_i}(S)
        &=
        \rdeg_{G[A]/\mathcal{A}_i}(S)+\rdeg_{G[A,B]/\mathcal{A}_s}(S) 
        \le
        \lambda_2+\lambda_1.
    \end{align*}\end{linenomath*}
    If $i>s$ and $S\subseteq A$, then by~\ref{item:expA} and~\ref{item:expAB},
    \begin{linenomath*}\begin{align*}
        \rdeg_{G/\Pi_i}(S)
        &\le\rdeg_{G/\Pi_i'}(S)+3\ell\\
        &=\rdeg_{G[A]/\mathcal{A}_s}(S)+\rdeg_{G[A,B]/(\mathcal{A}_s \cup \mathcal{B}_{i-s})}(S)+3\ell 
        \le\lambda_2'+\mu_2+3\ell.
    \end{align*}\end{linenomath*}
    If $i>s$ and $S\subseteq B$, then 
    by~\ref{item:expBA} and~\ref{item:expB}, 
    \begin{linenomath*}\begin{align*}
        \rdeg_{G/\Pi_i}(S)
        &\le\rdeg_{G/\Pi_i'}(S)+3\ell \\
        &=\rdeg_{G[A,B]/(\mathcal{A}_s \cup \mathcal{B}_{i-s})}(S)+\rdeg_{G[B]/\mathcal{B}_{i-s}}(S)+3\ell
        \le\rho_2+\nu_2+3\ell.
    \end{align*}\end{linenomath*}

    \medskip
    \noindent\textbf{Case III.} $\abs{S}=3$.
    
    Then $S \subseteq B$ and $i>s$.
    By~\ref{item:expBA} and~\ref{item:expB}, 
    \begin{linenomath*}\begin{align*}
        \rdeg_{G/\Pi_i}(S)
        &\le
        \rdeg_{G/\Pi_i'}(S)+3\ell \\
        &=
        \rdeg_{G[A,B]/(\mathcal{A}_s \cup \mathcal{B}_{i-s})}(S)+\rdeg_{G[B]/\mathcal{B}_{i-s}}(S)+3\ell
        \le
        \rho_3+\nu_3+3\ell.
    \end{align*}\end{linenomath*}

    \medskip
    \noindent\textbf{Case IV.} $\abs{S}=4$.
    
    Then $S \subseteq B$ and $i>s$.
    By~\ref{item:expB},
    \begin{linenomath*}\begin{align*}
        \rdeg_{G/\Pi_i}(S)
        &\le
        \rdeg_{G/\Pi_i'}(S)+3\ell \\
        &\le 
        (|A|-s)
        +
        \rdeg_{G[B]/\mathcal{B}_{i-s}}(S)
        +
        3\ell
        =
        (n-m-s)+\nu_4+3\ell.
    \end{align*}\end{linenomath*}
    
    This completes the proof.
\end{proof}

\subsection{Preliminaries for Lemma~\ref{lem:nonrandom}}\label{subsec:random}

As we mentioned in the overview of Section~\ref{sec:upper}, we will prove Theorem~\ref{thm:upper2} by using Lemma~\ref{lem:nonrandom}.
To do this, in this subsection, we examine several upper bounds for the numbers $\lambda_1$, $\lambda_2$, $\lambda_2'$, $\lambda_3$, $\mu_1$, $\mu_2$, $\rho_2$, $\rho_3$, $\nu_2$, $\nu_3$, and $\nu_4$ in the statement of Lemma~\ref{lem:nonrandom}.
These upper bounds correspond to conditions \ref{item:expABpairs}--\ref{item:expB} in the lemma and will be used to verify Theorem~\ref{thm:upper2} in Subsection~\ref{subsec:upper}.

We start with some properties of $\alpha(x)$ and $\beta(x)$.
The following lemma is proved by straightforward calculations.
We include its proof in Appendix~\ref{app:inequalities}.

\begin{restatable}{lemma}{lemab}\label{lem:ab}
    For all $x\in(0,1/2]$, the following hold.
    \begin{enumerate}[label=\rm(\roman*)]
        \item\label{lem:ab ab} For all $0<y<x$, we have that $\alpha(y)>\alpha(x)$ and $\beta(y)<\beta(x)$.
        \item\label{lem:ab p*} There is a unique $p^*\in(0,1/2]$ with $\alpha(p^*)=\beta(p^*)$.
        Moreover, $p^* \in(0.4012,0.4013)$.
        \item\label{lem:ab long ineq}
            The following three values have the same sign (positive, zero, negative): 
        \begin{linenomath*}\begin{align*}
            &p^*-x, \quad \alpha(x)-\beta(x),\quad \text{ and }\\
            &(1-x^8-(1-x)^8)(3\alpha(x)-1)+(1-x^{12}-(1-x)^{12})(1-2\alpha(x))-2x(1-x).
        \end{align*}\end{linenomath*}
        \item\label{lem:ab 1-alpha} $(1-\alpha(x))/2<2x(1-x)$.
    \end{enumerate}
\end{restatable}

For a constant $p\in[0,1]$ and two disjoint sets $A$ and $B$, let $G(A,B,p)$ be a random bipartite graph on $A\cup B$ such that two vertices $a\in A$ and $b\in B$ are joined by an edge with probability $p$ independently.

We will use the following lemma to choose $\lambda_1$ in~Lemma~\ref{lem:nonrandom}\ref{item:expABpairs}.

\begin{lemma}\label{lem:pairs general}
    Let $p \in[0.4,0.5]$, $q:= 1-p$, and $\varepsilon$ and $\delta$ be positive constants less than $0.01$.
    Let $n\ge1$ be an integer and $m:=n-\lfloor n^{1-\delta}\rfloor$.
    For $G:=G(A,B,p)$ with $\abs{A}=\lfloor n^{1-\delta} \rfloor$ and $\abs{B}=m$, with high probability, $G$ has at least $n^{(1+\varepsilon)/2}$ pairwise disjoint pairs of vertices $u,v\in A$ with $r_G(u,v)\le2pqm-c\sqrt{m\ln m}$, where $c=\sqrt{2pq(1-2pq)(3-3\varepsilon-4\delta)}$.
\end{lemma}

We note that to derive Theorem~\ref{thm:upper2}, we will use Lemma~\ref{lem:pairs general} applied to $p\in(p^*,1/2]$.
It is however convenient to prove Lemma~\ref{lem:pairs general} for a larger interval $[0.4,0.5]$.

\begin{proof}
    For each $\{u,v\}\in\binom{A}{2}$, let $X_{u,v}$ be the indicator random variable such that $X_{u,v}=1$ if $r_G(u,v)\le2pqm-c\sqrt{m\ln m}$, and $X_{u,v}=0$ otherwise.
    For each $u\in A$, let $X_u:=\sum_{v\in A\setminus\{u\}}X_{u,v}$.
    Let $X:=\sum_{\{u,v\}\in\binom{A}{2}}X_{u,v}$.
    
    \begin{claim}\label{cl:1}
        $\Pr[X \le n^{1/2+\varepsilon}] \to 0$ as $n \to \infty$.
    \end{claim}
    \begin{subproof}
        Let $\mu:=\E[X]$, $\sigma^2:=\Var[X]$, and $\nu := \E[X_{u,v}]$ for any $\{u,v\}\in\binom{A}{2}$.
        First, observe that $r_G(u,v)$ is precisely $\mathcal{B}(m,2pq)$.
        By Lemmas~\ref{lem:binomUpper} and~\ref{lem:binomLower},
        \begin{linenomath*}\begin{align*}
            \nu
            &= \Pr\left[\mathcal{B}(m,2pq)\le\left(2pq-c\sqrt{\frac{\ln m}{m}}\right)m\right]\\
            &
           =\exp\left( -\frac{c^2}{4pq(1-2pq)} \ln m+o(\ln m)\right)
           =n^{-\frac{3-3\varepsilon-4\delta}{2}+o(1)}, \\
            \mu
            &=\binom{\abs{A}}{2}\nu
            =\binom{\lfloor n^{1-\delta} \rfloor}{2} \nu
           =n^{\frac{1+3\varepsilon}{2}+o(1)} 
            \ge(2+\sqrt2) n^{1/2+\varepsilon}
            \text{ for sufficiently large $n$}.
        \end{align*}\end{linenomath*}
        
        Let $u$, $v$, $v'$ be fixed distinct vertices in~$A$.
        By Chebyshev's inequality, for sufficiently large $n$, we have
        \begin{linenomath*}\[
            \Pr[X\le n^{1/2+\varepsilon}]=\Pr[\mu-X\ge\mu-n^{1/2+\varepsilon}]\le\frac{\sigma^2}{(\mu-n^{1/2+\varepsilon})^2}
            \le\frac{2 \sigma^2}{\mu^2}.
        \]\end{linenomath*}
        We compute an upper bound for $\sigma^2$.
        Note that for distinct $a,a',b,b'\in A$, the random variables $X_{a,b}$ and $X_{a',b'}$ are independent.
        Then
        \begin{linenomath*}\begin{align*}
            \sigma^2
            &=
            \E[X^2]-\mu^2 \\
            &=
            \sum_{\{a,b\}\in {\binom{A}{2}}} \E[X_{a,b}]
           +\sum_{\substack{\{a,b\},\{a,b'\}\in\binom{A}{2}\\ b\neq b'}} \E[X_{a,b} X_{a,b'}]
           +\sum_{\substack{\{a,b\},\{a',b'\}\in\binom{A}{2}\\ \{a,b\}\cap \{a',b'\}=\emptyset}}\E[X_{a,b} X_{a',b'}]
           -\binom{|A|}{2}^2 \nu^2 \\
            &=
           \binom{\abs{A}}{2}\nu 
           +\abs{A}(\abs{A}-1)(\abs{A}-2)\E[X_{u,v}X_{u,v'}]
           +\binom{|A|}{2}\binom{|A|-2}{2}\nu^2
            -\binom{\abs{A}}{2}^2\nu^2\\
            &\le\binom{\abs{A}}{2}\nu 
           +n^{3(1-\delta)} \E[X_{u,v}X_{u,v'}]
            \\
            &= \mu+n^{3(1-\delta)} \E[X_{u,v} X_{u,v'}].
        \end{align*}\end{linenomath*}
        Thus, $\Pr[X\le n^{1/2+\varepsilon}]\le2\mu^{-1}+2\mu^{-2} n^{3(1-\delta)} \E[X_{u,v} X_{u,v'}]$.
        
        Since $\mu^{-1}=n^{-(1+3\varepsilon)/2+o(1)} \to 0$,
        it suffices to show that $\mu^{-2} n^{3(1-\delta)} \E[X_{u,v} X_{u,v'}] \to 0$.
        To do this, we now compute an upper bound for $\E[X_{u,v}X_{u,v'}]$.
        Let $c_1=p-\sqrt{\frac{4pq\ln m}{m}}$ and $c_2=p-\sqrt{\frac{2pq\ln m}{m}}$.
        We have
        \begin{linenomath*}\begin{equation}
        \begin{aligned}
            \E[X_{u,v}X_{u,v'}]
            &= \sum_{k=0}^{m} \E[X_{u,v}X_{u,v'} \mid \deg_{G}(u)=k]\cdot\Pr[\deg_{G}(u)=k]\\
            &\le\Pr[\deg_{G}(u)<c_1m]\\
            &\quad+\left(\max_{c_1m\le k<c_2m}\E[X_{u,v}X_{u,v'} \mid \deg_{G}(u)=k] \right)\Pr[c_1m \le\deg_G(u)<c_2m]\\
            &\quad+\left(\max_{k\ge c_2m}\E[X_{u,v}X_{u,v'} \mid \deg_{G}(u)=k]\right)\Pr[\deg_G(u)\ge c_2m].
        \end{aligned}
        \label{spadesuit}
        \end{equation}\end{linenomath*}
    By Lemma~\ref{lem:binomUpper}, we have
    \begin{linenomath*}\begin{align}
        \Pr[\deg_{G}(u)<c_1 m]&\le\exp(-2\ln m+o(\ln m))=n^{-2+o(1)}, \label{eq:low} \\
        \Pr[\deg_{G}(u)<c_2 m]&\le\exp(-\ln m+o(\ln m))=n^{-1+o(1)}.\label{eq:medium}
    \end{align}\end{linenomath*}

        Let $S$ be a subset of $B$ with $\abs{S}=\left(p-\gamma \sqrt{\frac{\ln m}{m}}\right)m$ and 
        $\gamma \le\sqrt{4pq}$.
        Let
        \begin{linenomath*}\[
            x:=\frac{\E[ \abs{N_G(v)\setminus S}]+\E[\abs{S\setminus N_G(v)}]}{m} =
            \frac{p\abs{B\setminus S}+q\abs{S}}{m}=2pq-(q-p)\gamma \sqrt{\frac{\ln m}{m}}.
        \]\end{linenomath*}
        Then $x(1-x)\le2pq(1-2pq)$.
        Since 
        \begin{linenomath*}\begin{equation}
            c \ge\sqrt{5.86 pq (1-2pq)}>\sqrt{4pq(1-4pq)}=(q-p)\sqrt{4pq} \ge(q-p)\gamma,
            \label{eq:estimatec}
        \end{equation}\end{linenomath*}
        we deduce the following bound from Lemma~\ref{lem:binomUpper}.
        \begin{linenomath*}\begin{equation}
        \begin{aligned}
            \E[X_{u,v} \mid N_{G}(u)=S]
            &=\Pr[r_{G}(u,v)\le2pqm-c\sqrt{m \ln m}\mid N_G(u)=S]\\
            &=\Pr[\abs{N_G(v)\setminus S}+\abs{S \setminus N_G(v)} \le2pqm-c\sqrt{m \ln m}]\\
            &\le\exp\left(-\frac{(c-(q-p)\gamma)^2}{2x(1-x)}\ln m+o(\ln m)\right)\\
            &\le\exp\left(-\frac{(c-(q-p)\gamma)^2}{4pq(1-2pq)}\ln m+o(\ln m)\right)
            =n^{-\frac{(c-(q-p)\gamma)^2}{4pq(1-2pq)}+o(1)}.
            \label{diamondsuit}
        \end{aligned}
        \end{equation}\end{linenomath*} 
        Since $p\in[0.4,0.5]$, we have that 
        \begin{linenomath*}\begin{align*} 
            (\sqrt{5.86}-\sqrt{2})\sqrt{1-2pq} &\ge(\sqrt{5.86}-\sqrt{2})\sqrt{1/2} >0.4 \ge2(q-p),\\
            (\sqrt{5.86}-2)\sqrt{1-2pq} &\ge(\sqrt{5.86}-2)\sqrt{1/2}>\sqrt{2} \cdot (0.2)\ge\sqrt{2}(q-p).
        \end{align*}\end{linenomath*}
        Thus, by~\eqref{eq:estimatec}, 
        \begin{linenomath*}\begin{align}
            (c-(q-p)\sqrt{4pq})^2 &\ge pq(\sqrt{5.86(1-2pq)}-2(q-p))^2>2pq(1-2pq),\label{eq:est-4pq}\\
            (c-(q-p)\sqrt{2pq})^2 &\ge pq(\sqrt{5.86(1-2pq)}-\sqrt{2}(q-p))^2> 4pq(1-2pq).\label{eq:est-2pq}
        \end{align}\end{linenomath*}
        Two random variables $X_{u,v}$ and $X_{u,v'}$ are conditionally independent given the event that $N_{G}(u)=S$.
        Furthermore, a function $f(\gamma):=n^{-\frac{(c-(q-p)\gamma)^2}{4pq(1-2pq)}}$ on $(-\infty,\sqrt{4pq}]$ is increasing by~\eqref{eq:estimatec} and thus it is maximized when $\gamma=\sqrt{4pq}$.
        Hence, we deduce by~\eqref{eq:medium},~\eqref{diamondsuit}, and~\eqref{eq:est-4pq}
        \begin{linenomath*}\begin{align*}
            \lefteqn{\left(\max_{c_1m\le k<c_2m}\E[X_{u,v}X_{u,v'} \mid \deg_{G}(u)=k] \right)
            \Pr[c_1m \le\deg_G(u)<c_2m]} \\
            &\le
            \left(\max_{k\ge c_1m}\E[X_{u,v}X_{u,v'} \mid \deg_{G}(u)=k] \right)
            \Pr[\deg_G(u)<c_2m]\\
            &= 
            \left(\max_{k\ge c_1m}\E[X_{u,v}\mid \deg_{G}(u)=k] \right)
            \left(\max_{k\ge c_1m}\E[X_{u,v'}\mid \deg_{G}(u)=k] \right)
            \Pr[\deg_G(u)<c_2m]\\
            &\le n^{-\frac{(c-(q-p)\sqrt{4pq})^2}{2pq(1-2pq)}-1+o(1)}\le n^{-2+o(1)}.
        \end{align*}\end{linenomath*}
        Similarly, by \eqref{diamondsuit} and \eqref{eq:est-2pq},
        \begin{linenomath*}\[
            \max_{k\ge c_2m}\E[X_{u,v} X_{u,v'}\mid \deg_{G}(u)=k] 
            \le n^{-\frac{(c-(q-p)\sqrt{2pq})^2}{2pq(1-2pq)}+o(1)}
            \le n^{-2+o(1)}.
        \]\end{linenomath*}
        The last two inequalities along with~\eqref{spadesuit} and~\eqref{eq:low} imply that $\E[X_{u,v}X_{u,v'}] \le 3n^{-2+o(1)}$.
        Therefore,
        \begin{linenomath*}\[
            \mu^{-2} n^{3(1-\delta)} \E[X_{u,v} X_{u,v'}] \le n^{2-3\varepsilon-3\delta+o(1)} \E[X_{u,v} X_{u,v'}]
            \le n^{-3\varepsilon-3\delta+o(1)} \to 0,
        \]\end{linenomath*}
        and this proves the claim.
    \end{subproof}

    \begin{claim}\label{cl:2}
        Let $\xi>0$ be a constant.
        With high probability, $\max\{X_u:u\in A\}\le n^\xi$.
    \end{claim}
    \begin{subproof}
        Let $c_2=p-\sqrt{\frac{2pq\ln m}{m}}$.
        Similar to \eqref{spadesuit}, for each $u\in A$,
        \begin{linenomath*}\begin{equation}\label{eq:summands}
            \Pr[X_u>n^\xi] \le\Pr[\deg_G(u)<c_2m]+\max_{k\ge c_2m}\Pr[X_u>n^\xi \mid \deg_G(u)=k].
        \end{equation}\end{linenomath*}
        Let $S$ be a subset of $B$ with $\abs{S}=\left(p-\gamma \sqrt{\frac{\ln m}{m}}\right)m$ and $\gamma \le\sqrt{2pq}$.
        Conditional on the event $N_G(u)=S$, the random variables $X_{u,v}$ for $v\in A\setminus \{u\}$
        are independent, identically distributed random variables
        with the Bernoulli distribution with mean $\nu' := \E[X_{u,v}\mid N_G(u)=S]$.
        Recall that
        \begin{linenomath*}\[
            \nu'\le n^{-\frac{(c-(q-p)\sqrt{2pq})^2}{4pq(1-2pq)}+o(1)}\le n^{-1+o(1)}
        \]\end{linenomath*}
        by~\eqref{diamondsuit} and \eqref{eq:est-2pq}.
        Hence, $n^\xi>6\nu'(\lfloor n^{1-\delta} \rfloor-1)$.
        By Lemma~\ref{lem:Chernoffbound}\ref{cond:chernoff2},
        \begin{linenomath*}\[ 
            \Pr[X_u>n^\xi \mid N_G(u)=S]
            =\Pr\left[ \mathcal{B}(\lfloor n^{1-\delta} \rfloor -1, \nu')
           >n^\xi \right]
            \le2^{-n^\xi}.
        \]\end{linenomath*}
        This together with~\eqref{eq:medium} and~\eqref{eq:summands} implies that for each $u\in A$, 
        $\Pr[X_u>n^\xi]\le n^{-1+o(1)}+2^{-n^{\xi}}$.
        By the union bound, $\Pr[X_u>n^\xi\text{ for some }u\in A] \le n^{1-\delta} \big( n^{-1+o(1)}+2^{-n^{\xi}} \big)\to0$ as $n\to\infty$.
    \end{subproof}
    
    By Claim~\ref{cl:1}, $X>n^{1/2+\varepsilon}$ with high probability.
    Applying Claim~\ref{cl:2} to $\xi=0.99 \varepsilon / 2$, we have that $X_u\le n^\xi$ for every $u\in A$ with high probability.
    We repeatedly choose a pair $(u,v)$ such that $r_G(u,v)\le2pqm-c\sqrt{m\ln m}$ and extract $u,v$ and all $w$ such that $\max\{r_G(u,w), r_G(v,w)\} \le2pqm-c\sqrt{m \ln m}$.
    Since $2 n^\xi n^{(1+\varepsilon)/2}<n^{1/2+\varepsilon}$, with high probability we can obtain at least $n^{(1+\varepsilon)/2}$ disjoint pairs $(u,v)$ of vertices in $A$ satisfying $r_G(u,v)\le2pqm-c\sqrt{m \ln m}$.
\end{proof}

The next lemma will be used to select~$\mu_1$ and~$\mu_2$ in~Lemma~\ref{lem:nonrandom}\ref{item:expAB}.

\begin{lemma}\label{lem:exposingAB}
    Let $p \in[0,1]$, $q := 1-p$, and $\varepsilon$ and $\delta$ be positive reals with $\varepsilon<1/2$ and $\delta<1$.
    Let $n\ge1$ be an integer and $m:=n-\lfloor n^{1-\delta}\rfloor$.
    Let $G=G(A,B,p)$ with $\abs{A}=\lfloor n^{1-\delta}\rfloor$ and $\abs{B}=m$, and $\Pi$ be a partition of $B$.
    With very high probability, for every nonempty subset $S \subseteq A$ with $\abs{S}\le10$, $\rdeg_{G/(\{S\}\cup\Pi)}(S)\le\abs{S} pqm+n^{1/2+\varepsilon}$.
\end{lemma}
\begin{proof}
    We may assume that $pq\neq0$.
    Let $m' := |\Pi|$.
    Then $m' \le m \le n$.
    We denote by $\Pi=\{T_1,T_2,\ldots,T_{m'}\}$ and $t_i := |T_i|$ for each $i \in[m']$.
    For each nonempty subset $S\subseteq A$, the probability that $S$ and $T_i$ are completely joined by edges or there are no edges between them in $G$ is $p^{\abs{S}t_i}+q^{\abs{S}t_i}$.

    Note that $(1-p^{\abs{S}t_i}-q^{\abs{S}t_i})\le\abs{S} t_i pq$ by Lemma~\ref{lem:pq}\ref{item:pq-2}.
    Therefore, we deduce that 
    \begin{linenomath*}\[
        \mu:=\E\left[\rdeg_{G/(\{S\} \cup \Pi)}(S)\right]
        =
        \sum_{i=1}^{m'} (1-p^{\abs{S}t_i}-q^{\abs{S}t_i})
        \le
        \sum_{i=1}^{m'} \abs{S}t_i pq=\abs{S}pqm.
    \]\end{linenomath*}
    Note that $\mu\le m'\le n$.
    We may assume that $\abs{\{i\in[m']: \abs{S}t_i\ge2\}}>pqm$, because otherwise $\rdeg_{G/(\{S\}\cup\Pi)}(S)\le\abs{\{i\in[m']: \abs{S}t_i\ge2\}}\le pqm \le\abs{S} pqm+n^{1/2+\varepsilon}$.
    This implies that 
    $\mu >pqm\cdot2pq$, 
    because $1-p^{x}-q^{x}\ge2pq$ for all $x\ge2$.

    We now bound from above the probability that $\rdeg_{G/(\{S\}\cup\Pi)}(S)$ is larger than $\abs{S} pqm+n^{1/2+\varepsilon}$.
    Note that $n^{1/2+\varepsilon}<2p^2q^2m<\mu$ for sufficiently large $n$.
    By Lemma~\ref{lem:Chernoffbound}\ref{cond:chernoff1}, we deduce the following.
    \begin{linenomath*}\begin{align*}
        \Pr\left[\rdeg_{G/(\{S\}\cup\Pi)}(S)>\abs{S}pqm+n^{1/2+\varepsilon}\right]
        &\le\Pr\left[\rdeg_{G/(\{S\} \cup \Pi)}(S)>\left(1+\frac{n^{1/2+\varepsilon}}{\mu}\right)\mu\right]\\
        &\le\exp\left(-\frac{n^{1+2\varepsilon}}{3\mu}\right)\le\exp\left(-\frac{n^{2\varepsilon}}{3}\right).
    \end{align*}\end{linenomath*}
    By Lemma~\ref{lem:Stirling}, the number of subsets of~$A$ with size at most~$10$ is at most $(e\abs{A}/10)^{10}\le n^{10(1-\delta)}$, and therefore the result now follows from the union bound.
\end{proof}

We will use the following lemma to select~$\rho_2$ and~$\rho_3$ in~Lemma~\ref{lem:nonrandom}\ref{item:expA}.

\begin{lemma}\label{lem:exposingA}
    Let $p\in[0,1]$, $q := 1-p$, and $\varepsilon\in(0,1/2)$.
    Let $n$ and $k$ be positive integers with $k \le n/2$, and $G := G(n,p)$.
    Let $\{u_1,v_1\},\ldots,\{u_k,v_k\}$ be disjoint pairs of vertices in~$G$, and let $G':=G/\{\{u_i,v_i\}:i\in[k]\}$.
    Then with very high probability, for all $i \in[k]$,
    \begin{linenomath*}\[
        \rdeg_{G'}(\{u_i,v_i\})\le2pqn-2p^2q^2 k+n^{1/2+\varepsilon}.
    \]\end{linenomath*}
\end{lemma}
\begin{proof}
    We may assume $pq\neq 0$.
    For every $i\in[k]$, let $A_i$ be the event that 
    $\rdeg_{G'}(\{u_i,v_i\})>2pqn-2p^2q^2 k+n^{1/2+\varepsilon}$.
    By Lemma~\ref{lem:pq}\ref{item:pq-2}, for all $i \in[k]$,
    \begin{linenomath*}\begin{align*}
        \mu:=\E[\rdeg_{G'}(\{u_i,v_i\})]
        &=2pq(n-2k)+(1-p^4-q^4)(k-1)\\
        &\le2pq(n-2)-2p^2q^2(k-1)\\
        &\le2pqn- 2p^2q^2 k.
    \end{align*}\end{linenomath*}
    Note that $\mu<n$.
    Note also that $1-p^4-q^4\ge2pq$ and therefore 
    $\mu\ge2pq(n-k+1)\ge pqn $, 
    implying that $n^{1/2+\varepsilon}<\mu$ for all sufficiently large $n$.
    By Lemma~\ref{lem:Chernoffbound}\ref{cond:chernoff1},
    \begin{linenomath*}\[ 
        \Pr[A_i]\le\exp\left(-\frac{n^{1+2\varepsilon}}{3\mu}\right)\le\exp\left(-\frac{1}{3}n^{2\varepsilon}\right).
    \]\end{linenomath*}
    By the union bound, $\Pr\left[\bigcup_{i=1}^kA_i\right]\le k\exp(-\Omega(n^{2\varepsilon}))=\exp(-\Omega(n^{2\varepsilon}))$.
\end{proof}

The upcoming lemma will be applied to choose~$\rho_2$ and~$\rho_3$ in~Lemma~\ref{lem:nonrandom}\ref{item:expBA}.
We remark that the set~$L$ in the following lemma will take the role of~$L$ in Lemma~\ref{lem:nonrandom}.

\begin{lemma}\label{lem:exposingBA}
    Let $p\in[0,1]$, $q:= 1-p$, and $\varepsilon$ and $\delta$ be positive reals less than $1/4$.
    Let $n\ge1$ be an integer and $m:=n-\lfloor n^{1-\delta} \rfloor$.
    Let $G=G(A,B,p)$ with $\abs{A}=\lfloor n^{1-\delta} \rfloor$ and $\abs{B}=m$, and $\Pi_B$ be a set of disjoint nonempty subsets of~$B$.
    With very high probability, for every set $\Pi_A$ of $\lfloor n^{1/2+\varepsilon} \rfloor$ disjoint pairs of vertices in $A$,
    there is $L \subseteq B$ with $\abs{L}<\sqrt{n}$ such that for all $S \in \Pi_B$ with $S \cap L=\emptyset$, we have 
    \begin{linenomath*}\[
        \rdeg_{G / (\Pi_A \cup \Pi_B)} (S)
        \le(1-p^{\abs{S}}-q^{\abs{S}})(n^{1-\delta}- 2\lfloor n^{1/2+\varepsilon}\rfloor )
       +(1-p^{2\abs{S}}-q^{2\abs{S}})\lfloor n^{1/2+\varepsilon}\rfloor+n^{1/2+\varepsilon/2}.
    \]\end{linenomath*}
\end{lemma}
\begin{proof}
    We may assume that $pq\neq 0$.
    Let $n_1 := |A|-
        2 \lfloor n^{1/2+\varepsilon}
        \rfloor
        \quad\text{and}\quad
        n_2 :=  \lfloor n^{1/2+\varepsilon} \rfloor$.
    Then $n_1+n_2 \le |A| \le n^{1-\delta}$.
    Let $\mathcal{A}$ be the collection of all sets of $n_2$ disjoint (unordered) pairs of vertices in $A$.
    We claim that with very high probability, for every $\Pi_A \in \mathcal{A}$ and for every choice of $t:=\lceil\sqrt{n}\rceil$ distinct elements $S_1,\ldots,S_t$ in $\Pi_B$,
    there is $i\in[t]$ such that
    \begin{linenomath*}\[
        \rdeg_{G/(\Pi_A\cup\Pi_B)}(S_i)\le(1-p^{\abs{S_i}}-q^{\abs{S_i}})n_1+(1-p^{2\abs{S_i}}-q^{2\abs{S_i}})n_2+n^{1/2+\varepsilon/2}.
    \]\end{linenomath*}
    
    First, let us explain why this claim would imply the lemma.
    Suppose the claim holds.
    Then with very high probability, for each $\Pi_A \in \mathcal{A}$, there are less than $t$ elements $T$ of $\Pi_B$ 
    such that 
    \begin{linenomath*}\[
        \rdeg_{G / (\Pi_A \cup \Pi_B)} (T)
       >(1-p^{\abs{T}}-q^{\abs{T}})n_1
       +(1-p^{2\abs{T}}-q^{2\abs{T}}) n_2+n^{1/2+\varepsilon/2}.
    \]\end{linenomath*}
    Let $L$ be a minimal subset of $B$ such that $L$ intersects all such elements $T$ of $\Pi_B$.
    Then $\abs{L}<t$ and for all $S\in \Pi_B$ with $S\cap L=\emptyset$, 
    \begin{linenomath*}\[
        \rdeg_{G/(\Pi_A \cup\Pi_B)}(S)\le(1-p^{\abs{S}}-q^{\abs{S}})n_1+(1-p^{2\abs{S}}-q^{2\abs{S}})n_2+n^{1/2+\varepsilon/2}, 
    \]\end{linenomath*}
    proving this lemma.
    
    Now let us prove the claim.
    Fix $\Pi_A \in\mathcal{A}$ and distinct nonempty sets $S_1,\ldots,S_t$ in $\Pi_B$.
    Denote $G' := G/(\Pi_A \cup \Pi_B)$.
    For each $i\in[t]$, let
    \begin{linenomath*}\[ 
        p_i':=1-p^{\abs{S_i}}-q^{\abs{S_i}},\quad\text{and}\quad p_i'':=1-p^{2\abs{S_i}}-q^{2\abs{S_i}},
    \]\end{linenomath*}
    and let $\mathcal{X}_i$ be the event that
    \(
        \rdeg_{G'}(S_i)>p_i' n_1+p_i'' n_2+n^{1/2+\varepsilon/2}
    \).
    We remark that $\mathcal{X}_1,\ldots,\mathcal{X}_t$ are mutually independent.

    Since $\Pi_A$ induces a partition of $A$ into $n_1$ singletons and $n_2$ pairs and $S_i$ is nonempty,
    \begin{linenomath*}\[
        \mu_i:=\E[\rdeg_{G'}(S_i)]=p_i'n_1+p_i''n_2.
    \]\end{linenomath*}
    Observe that $\mu_i\ge p_i''n_2\ge2pq(n^{1/2+\varepsilon}-1)>n^{1/2+\varepsilon/2}$ for all sufficiently large~$n$ and $\mu_i<n_1+n_2\le n^{1-\delta}$.
    Hence, by Lemma~\ref{lem:Chernoffbound}\ref{cond:chernoff1},
    \begin{linenomath*}\[ 
        \Pr[\mathcal{X}_i]=\Pr\left[\rdeg_{G'}(S_i)>\mu_i+n^{1/2+\varepsilon/2} \right]
        \le\exp\left(-\frac{n^{1+\varepsilon}}{3\mu_i}\right)
        \le\exp\left(-\frac{1}{3}n^{\varepsilon+\delta}\right).
    \]\end{linenomath*}
    Note that
    \begin{linenomath*}\[
        \Pr\left[\bigcap_{i=1}^{t}\mathcal{X}_i\right]=\prod_{i=1}^{t}\Pr[\mathcal{X}_i]
        \le\exp\left(-\frac{1}{3}t n^{\varepsilon+\delta}\right)
        \le\exp\left(-\frac{1}{3}n^{1/2+\varepsilon+\delta}\right).
    \]\end{linenomath*}
    Since $\abs{\Pi_B}\le\abs{B}\le n$, we have
    \begin{linenomath*}\[
       \binom{\abs{\Pi_B}}{t}\le\exp(t \ln n)=\exp(o(n^{1/2+\varepsilon+\delta})).
    \]\end{linenomath*}
    Also, because $|A| \le n$, we have
    \begin{linenomath*}\[ 
        \abs{\mathcal{A}}\le\binom{|A|}{2}^{n_2}\le\exp\left(2n_2\ln n \right)=\exp(o(n^{1/2+\varepsilon+\delta})).
    \]\end{linenomath*}
    Thus, by the union bound, the claim is proved.
\end{proof}

The following lemma applied to $\abs{S}=1$ will be used for $\lambda_3$ in Lemma~\ref{lem:nonrandom}\ref{item:expBsingles}.
Although we will apply the lemma only for the case that $S$ is a singleton, we provide a general statement with $S$ having an arbitrary size, since its size is not critical in the proof.

\begin{lemma}\label{lem:single red}
    Let $p \in[0,1]$, $q:= 1-p$, $G := G(n,p)$, $\varepsilon \in(0,1/2)$, and $\Pi$ be a partition of~$V(G)$.
    Then, with very high probability, for all $S\in \Pi$, 
    \begin{linenomath*}\[
        \rdeg_{G/\Pi}(S)\le\abs{S}pqn+n^{1/2+\varepsilon}.
    \]\end{linenomath*}
\end{lemma}
\begin{proof}
    We may assume that $pq\neq 0$.
    By Lemma~\ref{lem:pq}\ref{item:pq-2}, 
    \begin{linenomath*}\[ 
        \mu_S:=\E[ \rdeg_{G/\Pi}(S) ]
       =\sum_{T \in \Pi \setminus \{S\}} \left( 1-p^{\abs{S}\cdot|T|}-q^{\abs{S}\cdot|T|} \right)
        \le\abs{S}pq \sum_{T \in \Pi \setminus \{S\}} |T|
        \le\abs{S}pq n
    \]\end{linenomath*}
    and $\mu_S< n$.
    Since $ \rdeg_{G/\Pi}(S)\le\abs{\{T\in \Pi\setminus\{S\}: \abs{S}\abs{T}\ge2\}}$, we may assume that 
    $\abs{\{T\in \Pi\setminus\{S\}: \abs{S}\abs{T}\ge2\}}>pqn$
    and therefore $\mu_S>(2pq)(pqn)=2p^2q^2n$, implying that $\mu_S>n^{1/2+\varepsilon}$ for all sufficiently large~$n$.
    By Lemma~\ref{lem:Chernoffbound}\ref{cond:chernoff1},
    \begin{linenomath*}\begin{align*}
        \Pr[\rdeg_{G/\Pi}(S)>\abs{S}pqn+n^{1/2+\varepsilon}]
        \le\exp\left(-\frac{n^{1+2\varepsilon}}{3\mu_S}\right)
        \le\exp\left(-\frac{1}{3}n^{2\varepsilon}\right).
    \end{align*}\end{linenomath*}
    Since $\abs{\Pi}\le n$, 
    the result follows by the union bound.
\end{proof}

Finally, we provide a lemma for choosing $\nu_2$, $\nu_3$, and $\nu_4$ in Lemma~\ref{lem:nonrandom}\ref{item:expB}.
We remark that by Lemma~\ref{lem:ab}\ref{lem:ab ab}, $1/3<\alpha(0.5)\le\alpha(p)\le\alpha(0.4)<1/2 $ if $p\in[0.4,0.5]$.

\begin{lemma}\label{lem:exposingB}
    Let $p\in(p^*,1/2]$, $q:=1-p$, and $\alpha := \alpha(p)$.
    There exists $c' >0$ only depending on $p$ such that the following hold.
    Let $\varepsilon$ and $\delta$ be positive reals less than~$1$.
    Let $n\ge1$ be an integer and $m:=n-\lfloor n^{1-\delta} \rfloor$.
    Let $G := G(m,p)$ and for each integer $1\le i\le(m+\lfloor \alpha n \rfloor)/2$, let $\mathcal{B}_i := \mathcal{B}^{m,\lfloor \alpha n \rfloor}_i$.
    Then with very high probability, for every integer $1\le i\le(m+\lfloor \alpha n \rfloor)/2$ and $S \in \mathcal{B}_i$,
    \begin{enumerate}[label=\rm(\roman*)]
        \item if $\abs{S}=2$, then $\rdeg_{G/\mathcal{B}_i}(S)\le2pqm+n^{1/2+\varepsilon/2}$,
        \item if $\abs{S}=3$, then $\rdeg_{G/\mathcal{B}_i}(S)\le2pqn-3pqn^{1-\delta}+n^{1/2+\varepsilon/2}$, and
        \item if $\abs{S}=4$, then $\rdeg_{G/\mathcal{B}_i}(S)\le(2pq-c')n$.
    \end{enumerate}
\end{lemma}
\begin{proof}
    Fix a positive integer $i\le(m+\lfloor \alpha n \rfloor)/2$ and $S\in \mathcal{B}_i$.
    Each vertex of $G/\mathcal{B}_i$ is either an element of $\mathcal{B}_i$ or a singleton $\{v\}$ with $v \in V(G)\setminus \bigcup_{X \in \mathcal{B}_i} X$.
    Note that $\abs{V(G/\mathcal{B}_i)}= n-i \ge n-(m+\lfloor \alpha n \rfloor)/2 \ge n/4$, since $\alpha \le\alpha(0.4)<1/2$ by Lemma~\ref{lem:ab}.
    By Lemma~\ref{lem:pq}\ref{item:pq-2}, we have
    \begin{linenomath*}\begin{align*}
        \mu:=&\,\E[\rdeg_{G/{\mathcal{B}_i}}(S)]\\
        =&\left(\sum_{X\in V(G/\mathcal{B}_i)} (1-p^{\abs{S}\abs{X}}-q^{\abs{S}\abs{X}})\right)
       -(1-p^{\abs{S}^2}-q^{\abs{S}^2})<\abs{S}pqm.
    \end{align*}\end{linenomath*}
    Note that $\mu<m<n$.
    Observe that 
    $\mu\ge2pq (\abs{V(G/\mathcal{B}_i)}-1)\ge2pq(n/4-1)>n^{1/2+\varepsilon/2}$ for all sufficiently large~$n$.

    Since $\abs{S}\ge2$, by Lemma~\ref{lem:pqadd}, 
    for all positive integers $i$, we have 
    \begin{linenomath*}\begin{equation}
        \label{eq:increase}
        \sum_{X\in V(G/\mathcal{B}_{i+1})}(1-p^{\abs{S}\abs{X}}-q^{\abs{S}\abs{X}})
        <\sum_{X\in V(G/\mathcal{B}_i)}(1-p^{\abs{S}\abs{X}}-q^{\abs{S}\abs{X}}).
    \end{equation}\end{linenomath*}
    If $\abs{S}=2$, then by Lemma~\ref{lem:Chernoffbound}\ref{cond:chernoff1}, we deduce that
    \begin{linenomath*}\begin{align*}
        \Pr\left[\rdeg_{G/\mathcal{B}_i}(S)>2pqm+n^{1/2+\varepsilon/2}\right]
        &\le\Pr\left[\rdeg_{G/\mathcal{B}_i}(S)\ge\mu+n^{1/2+\varepsilon/2}\right]\\
        &\le\exp\left(-\frac{n^{1+\varepsilon}}{3\mu}\right)\le\exp\left(-\frac{1}{3} n^\varepsilon \right).
    \end{align*}\end{linenomath*}
    If $\abs{S}=3$, then 
    $i> \lfloor \alpha n \rfloor$ by the definition of $\mathcal{B}_i$.
    Note that $m-2\lfloor\alpha n\rfloor>0$ for all sufficiently large $n$ as $\alpha<1/2$.
    We have that
    \begin{linenomath*}\begin{align*}
        \mu &\le\sum_{X\in V(G/\mathcal{B}_{\lfloor\alpha n\rfloor+1})}(1-p^{\abs{S}\abs{X}}-q^{\abs{S}\abs{X}})-(1-p^9-q^9)\tag*{by~\eqref{eq:increase}}\\
        &=(1-p^3-q^3)(m-2\lfloor \alpha n\rfloor-1)+(1-p^6-q^6)(\lfloor\alpha n\rfloor-1)\\
        &<(1-p^3-q^3)(m-2\alpha n+1)+(1-p^6-q^6)(\alpha n-2)\\
        &\le(1-p^3-q^3)(m-2\alpha n)+(1-p^6-q^6)\alpha n-(1-p^6-q^6),
    \end{align*}\end{linenomath*}
    in which the second last inequality holds because $\lfloor\alpha n\rfloor>\alpha n-1$ and $(1-p^6-q^6)-2(1-p^3-q^3)<0$ by Lemma~\ref{lem:pqadd}.
    Recall that $1-p^3-q^3=3pq$ and
    \begin{linenomath*}\[
        \alpha=\alpha(p)=\frac{pq}{2(1-p^3-q^3)-(1-p^6-q^6)}=\frac{pq}{6pq-(1-p^6-q^6)}.
    \]\end{linenomath*}
    Then we have that 
    \begin{linenomath*}\begin{align*}
        \mu&\le(1-p^3-q^3)(m-2\alpha n)+(1-p^6-q^6)\alpha n-(1-p^6-q^6)\\
        &=(3pq(1-2\alpha)+(1-p^6-q^6)\alpha)n-3pq\lfloor n^{1-\delta}\rfloor-(1-p^6-q^6)\\
        &=2pqn-3pq\lfloor n^{1-\delta}\rfloor-(1-p^6-q^6)\tag*{by the definition of~$\alpha$}\\
        &\le2pqn-3pqn^{1-\delta}+3pq-(1-p^6-q^6)\\
        &\le2pqn-3pqn^{1-\delta}.
    \end{align*}\end{linenomath*}
    By Lemma~\ref{lem:Chernoffbound}\ref{cond:chernoff1},
    \begin{linenomath*}\begin{align*}
        \Pr\left[\rdeg_{G/\mathcal{B}_i}(S)>2pqn -3pqn^{1-\delta}+n^{1/2+\varepsilon/2}\right] 
        &\le\Pr\left[\rdeg_{G/\mathcal{B}_i}(S)\ge\mu+n^{1/2+\varepsilon/2}\right]\\
        &\le\exp\left(-\frac{n^{1+\varepsilon}}{3\mu}\right)
        \le\exp\left(-\frac{1}{3} n^{\varepsilon}\right).
    \end{align*}\end{linenomath*}
    Now it remains to consider the case that $\abs{S}=4$.
    Then $i>m-\lfloor\alpha n\rfloor$.
    We have
    \begin{linenomath*}\begin{align*}
        \mu&\le\sum_{X\in V(G/\mathcal{B}_{m-\lfloor\alpha n\rfloor+1})}(1-p^{\abs{S}\abs{X}}-q^{\abs{S}\abs{X}})-(1-p^{16}-q^{16})\tag*{by~\eqref{eq:increase}}\\
        &\le(1-p^8-q^8)(\lfloor\alpha n\rfloor-(m-2\lfloor\alpha n\rfloor)-2)+(1-p^{12}-q^{12})(m-2\lfloor\alpha n \rfloor)\\
        &\le(1-p^8-q^8)(3\alpha n -m-2)+(1-p^{12}-q^{12})(m-2\alpha n)\\
        &=((1-p^8-q^8)(3\alpha-1)+(1-p^{12}-q^{12})(1-2\alpha))n\\
        &\quad+((1-p^8-q^8)-(1-p^{12}-q^{12}))\lfloor n^{1-\delta} \rfloor
        -2(1-p^8-q^8)\\
        &\le((1-p^8-q^8)(3\alpha-1)+(1-p^{12}-q^{12})(1-2\alpha))n,
    \end{align*}\end{linenomath*}
    where the third inequality holds because $2(1-p^{12}-q^{12})<3(1-p^8-q^8)$ by Lemma~\ref{lem:pq}\ref{item:pq-1}
    and the last inequality holds because $1-p^8-q^8\le1-p^{12}-q^{12}$.
    By Lemma~\ref{lem:ab}\ref{lem:ab long ineq}, there exists some constant $c'>0$ that only depends on $p$ such that $c'\le pq/4$ and
    \begin{linenomath*}\[ 
        \mu\le((1-p^8-q^8)(3\alpha-1)+(1-p^{12}-q^{12})(1-2\alpha))n
        \le
        (2pq-2c')n.
    \]\end{linenomath*}
    Note that $\mu \ge2pq(n/4-1)>c'n$, as we assume $n$ is large.
    By Lemma~\ref{lem:Chernoffbound}\ref{cond:chernoff1},
    \begin{linenomath*}\begin{align*}
        \Pr\left[\rdeg_{G/\mathcal{B}_i}(S)>(2pq-c')n\right]
        &\le\Pr\left[\rdeg_{G/\mathcal{B}_i}(S)\ge\mu+c'n\right]\\
        &\le\exp\left(-\frac{(c'n)^2}{3\mu}\right)=\exp\left(-\frac{(c')^2 n}{3}\right).
    \end{align*}\end{linenomath*}
    The result follows by the union bound.
\end{proof}

\subsection{Proof of Theorem~\ref{thm:upper2}}\label{subsec:upper}

\begin{proof}[Proof of Theorem~\ref{thm:upper2}]
    Let $G:=G(n,p)$.
    We arbitrarily choose constants $\varepsilon$ and $\delta$ in $(0,0.005)$, and let
    \begin{linenomath*}\[
        c:=\sqrt{2pq(1-2pq)(3-6\varepsilon-4\delta)}.
    \]\end{linenomath*}
    By the arbitrary choices of $\varepsilon$ and $\delta$,
    it suffices to show that with high probability
    \begin{linenomath*}\[
        \tww(G)\le2pqn-c\sqrt{n\ln n}+o(\sqrt{n\ln n}).
    \]\end{linenomath*}
    
    Let $\alpha:=\alpha(p)$, $m:=n-\lfloor n^{1-\delta}\rfloor$, $s:=\lfloor n^{1/2+\varepsilon}\rfloor$, and $r:=\lfloor(m+\lfloor\alpha n\rfloor)/2\rfloor$.
    Let $B := [m]$ and $A := [n] \setminus [m]$ be disjoint subsets of $V(G)=[n]$.
    For each $i\in[r]$, let $\mathcal{B}_i := \mathcal{B}_i^{m,\lfloor \alpha n \rfloor}$, which is a partition of $B$.
    To deduce the upper bound for $\tww(G)$, we will apply Lemma~\ref{lem:nonrandom}.
    We check the conditions \ref{item:expABpairs}--\ref{item:expB} in Lemma~\ref{lem:nonrandom} as follows.
    
    \medskip

    (A)
    By Lemma~\ref{lem:pairs general} applied to $G[A,B]$
    with $2\varepsilon$ as $\varepsilon$, 
    with high probability, there are disjoint pairs $\{u_1,v_1\},\ldots,\{u_s,v_s\}$ of vertices in $A$ such that for every $i \in[s]$,
    \begin{linenomath*}\[
        r_{G[A,B]}(u_i,v_i)\le2pqm-c\sqrt{m \ln m}.
    \]\end{linenomath*}
    For each $i\in[s]$, let $\mathcal{A}_0 := \{\{w\} : w\in A\}$ and 
    \begin{linenomath*}\[
        \mathcal{A}_{i} := \left\{ \{u_j,v_j\} : j\in[i] \right\} \cup \left\{ \{w\} : w \in A \setminus \bigcup_{j=1}^i \{u_j,v_j\} \right\}
    \]\end{linenomath*}
    be partitions of $A$.
    Note that $r_{G[A,B]}(u_i,v_i)=\rdeg_{G[A,B]/\mathcal{A}_s}(\{u_i,v_i\})$.
    
    \medskip
    
    (B)
    By Lemma~\ref{lem:exposingAB} applied to $G[A,B]$ and $\Pi=V(G/\mathcal{B}_i)$ with $i \in[r]$,
    we deduce that 
    with high probability, for each subset $S \subseteq A$ with $\abs{S} \in \{1,2\}$ and $i \in[r]$,
    \begin{linenomath*}\[ 
        \rdeg_{G[A\cup B]/(\{S\} \cup \mathcal{B}_i)}(S)\le\abs{S} pqm+n^{1/2+\varepsilon/2}.
    \]\end{linenomath*}
    
    \medskip
    
    (C)
    By the union bound and Lemma~\ref{lem:exposingA} applied to $G[A]$ and $\Pi=\mathcal{A}_i$ with $i\in[s]$, 
    we conclude that with high probability, for every $i,j \in[s]$ with $i \ge j$,
    \begin{linenomath*}\[ 
        \rdeg_{G[A]/\mathcal{A}_i}(\{u_j,v_j\})
        \le2pq n^{1-\delta}-2p^2q^2 i+(n^{1-\delta})^{(1+\delta)/2}
        \le2pq n^{1-\delta}+n^{1/2}.
    \]\end{linenomath*}
    In particular, we note that
    \begin{linenomath*}\[ 
        \rdeg_{G[A]/\mathcal{A}_s}(\{u_j,v_j\})
        \le2pq n^{1-\delta}-2p^2q^2 s+n^{1/2}.
    \]\end{linenomath*}
    
    \medskip
    
    (D)
    By Lemma~\ref{lem:exposingBA} applied to $G[A,B]$ and $\Pi_B \in \{\mathcal{B}_{\lfloor \alpha n \rfloor}, \mathcal{B}_{m-\lfloor \alpha n \rfloor} \}$, with high probability, there are subsets $L_1$ and $L_2$ of $B$ with $\max\{|L_1|, |L_2|\}<\sqrt{n}$ such that for every $S \in \mathcal{B}_{i}$ with $i \in \{\lfloor \alpha n \rfloor, m-\lfloor \alpha n \rfloor\}$
    and $S\cap L_j=\emptyset$ for each $j\in\{1,2\}$, we have 
    \begin{linenomath*}\[
        \rdeg_{G[A,B] / (\mathcal{A}_s \cup \mathcal{B}_i)} (S)
        \le
        (1-p^{\abs{S}}-q^{\abs{S}})(n^{1-\delta}-2s)+(1-p^{2\abs{S}}-q^{2\abs{S}}) s+n^{1/2+\varepsilon/2}.
    \]\end{linenomath*}
    Let $L := L_1 \cup L_2$.
    Then $|L| <2 \sqrt{n}$.
    Observe that $G[A,B]$ have no edges between vertices in $B$
    and so if $S\in \mathcal{B}_i \cap \mathcal{B}_j$, then 
    $\rdeg_{G[A,B]/(\mathcal A_s\cup \mathcal{B}_i)}(S)=\rdeg_{G[A,B]/(\mathcal A_s\cup \mathcal{B}_j)}(S)$.
    Note that for every $j\in[r]$ and every $S \in \mathcal{B}_j$, if $\abs{S}=2$, then 
    $S\in \mathcal{B}_{\lfloor \alpha n \rfloor}$ 
    and if $\abs{S}=3$, then $S\in\mathcal{B}_{m-\lfloor\alpha n\rfloor}$.
    Therefore, with high probability, for every $j\in[r]$ and $S \in \mathcal{B}_j$ with $\abs{S} \in \{2,3\}$ and $S\cap L=\emptyset$,
    \begin{linenomath*}\[
        \rdeg_{G[A,B] / (\mathcal{A}_s \cup \mathcal{B}_j)} (S)
        \le 
        (1-p^{\abs{S}}-q^{\abs{S}})(n^{1-\delta}-2s)+(1-p^{2\abs{S}}-q^{2\abs{S}}) s+n^{1/2+\varepsilon/2}.
    \]\end{linenomath*}
    
    \medskip
    
    (E)
    By the union bound and Lemma~\ref{lem:single red} applied to $G[B]$ and $\Pi_{i,u}=\{S \in \mathcal{B}_i : u \notin S\}$ with $u\in B$ and $i\in[r]$, 
    with high probability, for every $u\in B$ and $i\in[r]$, we have 
    \begin{linenomath*}\[
        \rdeg_{G[B]/\{S \in \mathcal{B}_i : u \notin S\}}(\{u\})
        \le pqn+n^{1/2+\varepsilon}.
    \]\end{linenomath*}
    
    \medskip
    
    (F)
    By Lemma~\ref{lem:exposingB} applied to $G[B]$, there exists $c'>0$
    depending only on $p$ 
    such that with high probability, for each $i\in[r]$ and $S \in \mathcal{B}_i$ with $\abs{S}>1$, we have 
    \begin{linenomath*}\[
        \rdeg_{G[B]/\mathcal{B}_i} (S)
        \le
        \begin{cases}
            2pqm+n^{1/2+\varepsilon/2} & \text{if } \abs{S}=2, \\
            2pqn-3pqn^{1-\delta}+n^{1/2+\varepsilon/2} & \text{if } \abs{S}=3, \\
            (2pq-c')n & \text{if } \abs{S}=4.
        \end{cases}
    \]\end{linenomath*}
    \medskip
    To complete the proof, we apply Lemma~\ref{lem:nonrandom} to
    \begin{linenomath*}\begin{align*}
        \ell        &= 2\sqrt{n}, \\
        \lambda_1   &= 2pqm-c \sqrt{m \ln m}, \\
        \lambda_2   &= 2pqn^{1-\delta}+n^{1/2},\\
        \lambda_2'  &= 2pqn^{1-\delta}-2p^2q^2s+n^{1/2},\\
        \lambda_3   &= pqn+n^{1/2+\varepsilon},\\
        \mu_1       &= pqm+n^{1/2+\varepsilon/2},\\
        \mu_2       &= 2pqm+n^{1/2+\varepsilon/2},\\
        \rho_2      &= 2pq(n^{1-\delta}-2s)+(1-p^4-q^4)s+n^{1/2+\varepsilon/2},\\
        \rho_3      &= 3pq(n^{1-\delta}-2s)+(1-p^6-q^6)s+n^{1/2+\varepsilon/2},\\
        \nu_2       &= 2pqm+n^{1/2+\varepsilon/2},\\
        \nu_3       &= 2pqn-3pqn^{1-\delta}+n^{1/2+\varepsilon/2}, \\
        \nu_4       &= (2pq-c')n.
    \end{align*}\end{linenomath*}
    Then it suffices to show that the following eight numbers
    \begin{linenomath*}\begin{align*}
        \begin{array}{llll}
            \textbf{(a)}\; n-s-r+3\ell,
            &\textbf{(b)}\;s+\mu_1,
            &\textbf{(c)}\;\lambda_1+\lambda_2,
            &\textbf{(d)}\;\lambda_2'+\mu_2+3\ell,\\
            \textbf{(e)}\;s+\lambda_3,
            &\textbf{(f)}\;\rho_2+\nu_2+3\ell,
            &\textbf{(g)}\;\rho_3+\nu_3+3\ell,
            &\textbf{(h)}\;n-m-s+\nu_4+3\ell
        \end{array}
    \end{align*}\end{linenomath*}
    are less than $2pqn-c\sqrt{n \ln n}+o(\sqrt{n \ln n})$.
    
    \leqnomode 

    Recall that $\ell=2\sqrt{n}$.
    Hence by Lemma~\ref{lem:ab}\ref{lem:ab 1-alpha} we have 
    \begin{linenomath*}\[ 
        \ltag{(a)}n-s-\lfloor (m+\lfloor \alpha n \rfloor)/2 \rfloor+3\ell
        =\frac{1-\alpha}{2}n+\frac{1}{2}n^{1-\delta}-n^{1/2+\varepsilon}+O(\sqrt{n})<(2pq-c'')n+o(n)
    \]\end{linenomath*}
    for some $c''>0$.
    Other bounds can be computed as follows.
    \begin{linenomath*}\begin{align*}
        \ltag{(b)} s+\mu_1
        &=\lfloor n^{1/2+\varepsilon}\rfloor+(pqm+n^{1/2+\varepsilon})
        \le pqn+o(n).\\
        \ltag{(c)}\lambda_1+\lambda_2
        &=(2pqm-c\sqrt{m \ln m})+(2pqn^{1-\delta}+n^{1/2})\\
        &=2pqn-c\sqrt{n \ln n}+o(\sqrt{n \ln n}).\\
        \ltag{(d)} \lambda_2'+\mu_2+3\ell
        &=(2pqn^{1-\delta}-2p^2q^2 \lfloor n^{1/2+\varepsilon}\rfloor+n^{1/2})
        +(2pqm+n^{1/2+\varepsilon/2})+6\sqrt{n}\\
        &=2pqn-2p^2q^2 n^{1/2+\varepsilon}+o(n^{1/2+\varepsilon})
        \le2pqn -c\sqrt{n\ln n}+o(\sqrt{n\ln n}).\\
        \ltag{(e)} s+\lambda_3
        &=\lfloor n^{1/2+\varepsilon}\rfloor+(pqn+n^{1/2+\varepsilon})=pqn+o(n).
    \end{align*}\end{linenomath*}
    By Lemma~\ref{lem:pq}\ref{item:pq-2},
    \begin{linenomath*}\begin{align*}
       \ltag{(f)} \rho_2+\nu_2+3\ell
        &=(2pq(n^{1-\delta}-2s)+(1-p^4-q^4)s+n^{1/2+\varepsilon/2})\\
        &\quad+(2pqm+n^{1/2+\varepsilon/2})+6\sqrt{n} \\
        &\le2pqn-2p^2q^2 n^{1/2+\varepsilon}+ o(n^{1/2+\varepsilon})
    \end{align*}\end{linenomath*}
    and
    \begin{linenomath*}\begin{align*}
        \ltag{(g)}\rho_3+\nu_3+3\ell
        &=(3pq(n^{1-\delta}-2s)+(1-p^6-q^6)s+n^{1/2+\varepsilon/2})\\
        &\quad+(2pqn-3pqn^{1-\delta}+n^{1/2+\varepsilon/2})+6\sqrt{n} \\
        &=2pqn-34p^3 q^3 n^{1/2+\varepsilon}+o(n^{1/2+\varepsilon}).
    \end{align*}\end{linenomath*}
    Finally,
    \begin{linenomath*}\[
       \ltag{(h)} n-m-s+\nu_4+3\ell
        =n^{1-\delta}-s+(2pq-c')n+6\sqrt{n}\le(2pq-c')n+o(n).
    \]\end{linenomath*}
    This completes the proof.
\end{proof}
\reqnomode

\section{Lower bounds for the twin-width of random graphs}\label{sec:lower}

We provide several lower bounds for the twin-width of random graphs $G(n,p)$ according to whether the random graphs are dense or sparse.
In Subsection~\ref{subsec:dense}, we prove the lower bound of Theorem~\ref{thm:main}\ref{item:main1}.
In Subsection~\ref{subsec:less than p star}, we prove Theorem~\ref{thm:main}\ref{item:main2}, which gives a better lower bound than Theorem~\ref{thm:main}\ref{item:main1} when $p<p^*$.
In Subsection~\ref{subsec:sparse}, we deal with the case when $p=o(1)$.

\subsection{When {{\boldmath$p$}} is a constant larger than {\boldmath$p^*$}} \label{subsec:dense}

As mentioned before, for a random graph $G := G(n,p)$ and distinct vertices $u$ and $v$, the expected size of $(N_G(u)\triangle N_G(v))\setminus \{u,v\}$ is $2p(1-p)(n-2)$.
This was the main idea used in the proof of Theorem~\ref{thm:ahko} in \cite{AHKO2021}.
However, it turns out that it is possible to improve the lower bound by an additive factor of $\Theta\left(\sqrt{n\log n}\right)$ by analyzing the first $\sqrt{n}$ contractions carefully.
This matches the upper bound up to the second-order term for $p\in(p^*,\frac{1}{2}]$.

The following simple lemma is a useful tool for computing lower bounds on twin-width.

\begin{lemma}\label{lem:tinygoodpairs}
    Let $b$ and $d$ be positive reals, and $H$ be a graph with $\abs{V(H)}\ge b+2$.
    If $H$ has fewer than $b$ pairs $\{u,v\}$ of distinct vertices with $r_H(u,v)<b+d$, then $\tww(H)>d$.
\end{lemma}
\begin{proof}
    We may assume that $V(H)=[n]$ for some integer $n$.
    Suppose for contradiction that there exists a $d$-contraction sequence $\sigma$ from $H_1:=H$ to the $1$-vertex graph $H_n$.
    For each $s\in[n-1]$, if $H_{s+1}$ is obtained from $H_s$ by contracting $\{i,j\}$ for some $1\le i<j\le n$, then we call the new vertex $i$.
    By definition, for any $s\in[n-1]$ and any distinct $i$ and $j$ in $V(H_{s+1})$, we have that $r_{H_{s+1}}(i,j)\ge r_{H_s}(i,j)-1$.
    For every positive integer $s$ less than $b+1$, we can find a pair $\{i,j\}$ with $1\le i<j\le n$ such that $H_{s+1}=H_s/\{i,j\}$ and
    \begin{linenomath*}\[
        r_H(i,j)\le r_{H_s}(i,j)+(s-1)=\rdeg_{H_{s+1}}(i)+(s-1)<b+d.
    \]\end{linenomath*}
    Thus, we have found $\lceil b\rceil$ pairs $\{i,j\}$ with $r_H(i,j)<b+d$, a contradiction.
\end{proof}

Theorem~\ref{thm:main}\ref{item:main1} directly follows from Theorem~\ref{thm:upper2} and the following theorem.

\begin{theorem}\label{thm:lower}
    For $p \in(0,1)$ and $q:=1-p$, with high probability,
    \begin{linenomath*}\[
        \tww(G(n,p))\ge2pqn-\sqrt{6pq(1-2pq)n\ln n}-o(\sqrt{n\ln n}).
    \]\end{linenomath*}
\end{theorem}
\begin{proof}
    Let $G:=G(n,p)$.
    Let $g(n)$ be any positive function on $\mathbb R$ such that 
    $g(n)/\sqrt{n}\to\infty$ and $g(n)/\sqrt{n\ln n}\to 0$ as $n\to \infty$.
    Let 
    \begin{linenomath*}\[
        f(n):=2pq(n-2)-\sqrt{6pq(1-2pq)(n-2)\ln n}-g(n).
    \]\end{linenomath*}
    Let $X$ be the number of pairs $\{x,y\}$ of distinct vertices with $r_G(x,y)<f(n)+g(n)$.
    We claim that
    \begin{linenomath*}\[
        \Pr\left[X \ge g(n)\right]\to 0
    \]\end{linenomath*}
    as $n\to\infty$.
    Let $u$, $v$ be distinct vertices.
    By Markov's inequality and the linearity of expectation,
    \begin{linenomath*}\[ 
        g(n)\cdot\Pr\left[X
        \ge g(n)\right]
        \le\E\left[X\right]
        =\binom{n}{2}\cdot\Pr\left[r_G(u,v)<f(n)+g(n)\right].
    \]\end{linenomath*}
    Note that
    \begin{linenomath*}\begin{align*}
        &\Pr\left[r_G(u,v)<f(n)+ g(n)\right]\\
        &=\Pr\left[\mathcal{B}(n-2, 2pq)<2pq(n-2)-\sqrt{6pq(1-2pq)(n-2)\ln n}\right]\\
        &=\Pr\left[\mathcal{B}(n-2,2pq)<2pq(n-2)-\left(\sqrt{\frac{6pq(1-2pq)\ln n}{n-2}}\right)(n-2)\right].
    \end{align*}\end{linenomath*}
    Let $\varepsilon(n):=\sqrt{\frac{6pq(1-2pq)\ln n}{n-2}}$.
    Note that $\varepsilon(n)\le3\cdot2pq/10$
    for all sufficiently large~$n$ because $\varepsilon(n)\to0$ as $n\to\infty$.
    Furthermore, $(n-2)\varepsilon(n)^3\le 8p^2q^2(1-2pq)^2$ for all sufficiently large $n$
    as $(n-2)\varepsilon(n)^3\to0$ as $n\to\infty$.
    Thus, by Lemma~\ref{lem:binomUpper}, for all sufficiently large $n$,
    
    \begin{linenomath*}\begin{align*}
        \Pr\left[r_G(u,v)<f(n)+g(n)\right]
        &\le\exp\left(-\frac{(n-2)\varepsilon(n)^2}{4pq(1-2pq)}+\frac{(n-2)\varepsilon(n)^3}{8p^2q^2(1-2pq)^2}\right)\\
        &\le\exp\left(-\frac{3}{2}\ln n+1\right)=en^{-3/2}.
    \end{align*}\end{linenomath*}
    Therefore,
    \begin{linenomath*}\begin{align*}
        \Pr\left[X\ge g(n)\right]
        &\le\frac{1}{g(n)}\binom{n}{2}\cdot\Pr[r_G(u,v)<f(n)+g(n)]\\
        &\le\frac{1}{g(n)}\binom{n}{2} en^{-3/2} 
       \to0\text{ as }n\to\infty.
    \end{align*}\end{linenomath*}
    Hence, by Lemma~\ref{lem:tinygoodpairs} applied to $b=g(n)$ and $d=f(n)$, 
    with high probability,
    \begin{linenomath*}\[
        \tww(G)>f(n)=2pqn-\sqrt{6pq(1-2pq)(n-2)\ln n}-g(n)-4pq,
    \]\end{linenomath*}
    which completes the proof.
\end{proof}

\subsection{When {{\boldmath$p$}} is a constant smaller than {\boldmath$p^*$}}\label{subsec:less than p star}

In this subsection, we prove Theorem~\ref{thm:main}\ref{item:main2}, which for convenience we restate as Theorem~\ref{thm:main2} below.
We begin by summarizing our strategy.

\medskip
\noindent\textbf{Overview.}
For an arbitrary partition sequence $\Pi_1,\Pi_2,\ldots,\Pi_n$ of~$[n]$, we first consider the minimum integer~$t$ such that 
$\abs{\bigcup_{X\in\Pi_t,\abs{X}\ge3}X} \ge \sqrt{n}$. 
At this point, we observe that if the expected width of~$\Pi_t$ with respect to $G(n,p)$ is not much larger than $2pqn$, then the number of parts of size~$2$ in~$\Pi_t$ is greater than the number of parts of size~$1$ in~$\Pi_t$ by some margin.
The key observation is that for every $t'>t$, this implies an upper bound on the number of parts of~$\Pi_{t'}$ of size~$3$.
We next consider the minimum integer $t'$ such that $\abs{\bigcup_{X\in\Pi_{t'},\abs{X}\ge4}X}\ge2\sqrt{n}$.
We are now able to prove a lower bound on the expected width of $\Pi_{t'}$ with respect to $G(n,p)$ in terms of the number of parts of size~$3$ in~$\Pi_{t'}$, which allows us to complete the proof.

\begin{theorem}\label{thm:main2}
    For $p \in(0,p^*)$ and $q := 1-p$, there exists a constant $c>0$ such that with high probability,
    \begin{linenomath*}\[ 
        \tww(G(n,p))>(2pq+c)n.
    \]\end{linenomath*}
\end{theorem}

\begin{proof}
    Let
    $\alpha := \alpha(p) >0$.
    We choose $\varepsilon$, $\delta$, and $c$ as follows:
    \begin{linenomath*}\begin{align*}
        2\varepsilon
        &:= (1-p^8-q^8)(3\alpha-1)+(1-p^{12}-q^{12})(1-2\alpha)-2pq, \\
        \delta
        &:=\min\left\{0.99\alpha,\frac{\varepsilon}{3(1-p^8-q^8)-2(1-p^{12}-q^{12})}\right\},\\
        1.01c
        &:= \min\{ \varepsilon, (2(1-p^3-q^3)-(1-p^6-q^6))\delta\}.
    \end{align*}\end{linenomath*}
    By Lemma~\ref{lem:ab}\ref{lem:ab long ineq}, $\varepsilon$ is positive.
    This together with Lemma~\ref{lem:pq}\ref{item:pq-1} implies that $\delta$ and $c$ are also positive.

    We claim that for an arbitrary partition sequence $\Pi_1,\ldots,\Pi_n$ of~$[n]$ and $G:=G(n,p)$, 
    \begin{linenomath*}\begin{equation}
        \Pr[\text{$\DeltaR(G/\Pi_{s})\le(2pq+c)n$ for all $s \in[n]$}] \le\exp\left(-\frac{c^2n^{3/2}}{400000}\right).\label{eq:stars}
    \end{equation}\end{linenomath*}
    Before proving \eqref{eq:stars}, we first show that \eqref{eq:stars} implies Theorem~\ref{thm:main2}.
    Assume that \eqref{eq:stars} holds.
    Lemma~\ref{lem:number-partition-sequence} says that there are at most $\exp(2n\ln n)$ partition sequences of $[n]$.
    Thus, by the union bound, the probability that the width of some partition sequence of $[n]$ 
    with respect to $G(n,p)$ is at most $(2pq+c)n$ is at most $\exp(2n\ln n-\Omega(n^{3/2}))$, which goes to $0$ as $n\to\infty$.
    Thus, $\tww(G(n,p))>(2pq+c)n$ with high probability by Lemma~\ref{lem:partionsequence}.

    We now prove \eqref{eq:stars}.
    For each $s\in[n]$ and an integer $i$, we denote by~$C_{s,i}$ the number of parts of $\Pi_s$ of size $i$.
    Let $t\in[n]$ be the smallest integer such that $\sum_{i=3}^n(i \cdot C_{t,i})\ge\sqrt{n}$.
    We consider two cases by comparing $C_{t,2}$ and $(\alpha-\delta)n$.

    Let~$X$ and~$Y$ be the two parts in $\Pi_{t-1}\setminus \Pi_t$.
    By the definition of~$t$, we have that
    \begin{linenomath*}\begin{align*}
        \sum_{i=3}^n(i\cdot C_{t,i})
        &\le\sum_{i=3}^n(i\cdot C_{t-1,i})+\min(2,\abs{X})+\min(2,\abs{Y})\\
        &\le\sum_{i=3}^n(i\cdot C_{t-1,i})+4 <\sqrt{n}+4.
    \end{align*}\end{linenomath*}
    By Observation~\ref{obs}, $\sum_{i=1}^n(i\cdot C_{t,i})=n$, so $C_{t,1}>n-2C_{t,2}-\sqrt{n}-4$.

    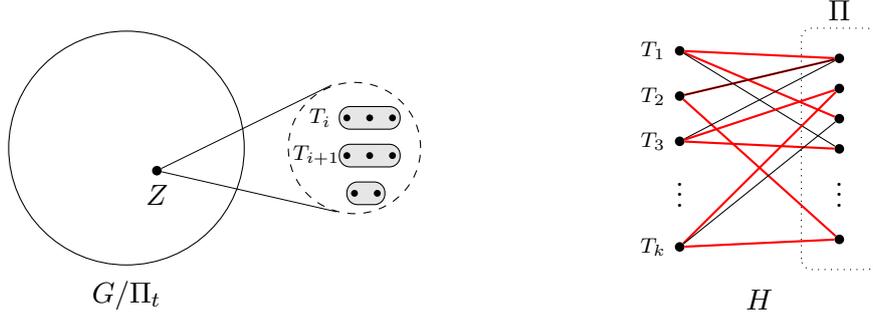
\begin{figure}
        \centering
        \begin{tikzpicture}
            \draw[color=black] (0,0) circle (1.55);
            \node[shape=circle,fill=black, scale=0.35] (v) at (0.4,-0.3) {};
            
            \draw[color=black, dashed] (3.0,0) circle (0.87);
            \draw[very thin, color=black] (v) -- (2.70,0.835);
            \draw[very thin, color=black] (v) -- (2.765,-0.845);
            \draw[rounded corners, fill=gray!20] (2.8,0.25) rectangle (3.6,0.55);
            \draw[rounded corners, fill=gray!20] (2.8,-0.25) rectangle (3.6,0.05);
            \draw[rounded corners, fill=gray!20] (2.9,-0.75) rectangle (3.4,-0.45);
            
            \node () at (0,-1.95) {$G/\Pi_t$};
            \node () at (0.4,-0.62) {$Z$};
            \node () at (2.55,0.4) {\scriptsize{$T_i$}};
            \node () at (2.50,-0.1) {\scriptsize{$T_{i+1}$}};
            \node[shape=circle,fill=black, scale=0.25] () at (2.9,0.4) {};
            \node[shape=circle,fill=black, scale=0.25] () at (3.2,0.4) {};
            \node[shape=circle,fill=black, scale=0.25] () at (3.5,0.4) {};
            \node[shape=circle,fill=black, scale=0.25] () at (2.9,-0.1) {};
            \node[shape=circle,fill=black, scale=0.25] () at (3.2,-0.1) {};
            \node[shape=circle,fill=black, scale=0.25] () at (3.5,-0.1) {};
            \node[shape=circle,fill=black, scale=0.25] () at (3.0,-0.6) {};
            \node[shape=circle,fill=black, scale=0.25] () at (3.3,-0.6) {};
        \end{tikzpicture}
        \hspace{2.5cm}
        \begin{tikzpicture}
            \node[shape=circle,fill=black, scale=0.35] (t1) at (0,2.9) {};
            \node () at (-0.35,2.9) {\scriptsize{$T_1$}};
            \node[shape=circle,fill=black, scale=0.35] (t2) at (0,2.3) {};
            \node () at (-0.35,2.3) {\scriptsize{$T_2$}};
            \node[shape=circle,fill=black, scale=0.35] (t3) at (0,1.7) {};
            \node () at (-0.35,1.7) {\scriptsize{$T_3$}};
            \node () at (0,1.1) {$\vdots$};
            \node[shape=circle,fill=black, scale=0.35] (tk) at (0,0.3) {};
            \node () at (-0.35,0.3) {\scriptsize{$T_k$}};
            
            \node () at (2.1,3.43) {$\Pi$};
            \draw[rounded corners, dotted] (1.6,0) rectangle (2.6,3.2);
            \node () at (1.05,-0.4) {$H$};
            
            \node[shape=circle,fill=black, scale=0.35] (p1) at (2.1,2.8) {};
            \node[shape=circle,fill=black, scale=0.35] (p2) at (2.1,2.4) {};
            \node[shape=circle,fill=black, scale=0.35] (p3) at (2.1,2.0) {};
            \node[shape=circle,fill=black, scale=0.35] (p4) at (2.1,1.6) {};
            \node () at (2.1,1.1) {$\vdots$};
            \node[shape=circle,fill=black, scale=0.35] (pk) at (2.1,0.4) {};
            
            \draw[red,thick] (t1) -- (p1);
            \draw[red,thick] (t1) -- (p3);
            \draw[black] (t1) -- (p4);
            \draw[red,thick] (t2) -- (p1);
            \draw[black] (t2) -- (p1);
            \draw[red,thick] (t2) -- (pk);
            \draw[black] (t3) -- (p1);
            \draw[red,thick] (t3) -- (p2);
            \draw[red,thick] (t3) -- (p4);
            \draw[red,thick] (tk) -- (p2);
            \draw[black] (tk) -- (p3);
            \draw[red,thick] (tk) -- (pk);
        \end{tikzpicture}
        \caption{The left illustrates the partition of~$Z$ into vertex sets of size~$3$ except one part, and the right is the bipartite trigraph~$H$ on bipartition $\{T_1,\ldots,T_k\}\cup\Pi$.}
        \label{fig:triplets}
    \end{figure}

    \medskip
    \noindent\textbf{Case I.} $C_{t,2}\le(\alpha-\delta)n$.

    For each part $Z$ of $\Pi_t$ of size at least~$3$, 
    we partition~$Z$ into $\left\lfloor\frac{\abs{Z}}{3}\right\rfloor$ parts of size $3$ and at most one part of smaller size; see Figure~\ref{fig:triplets}.
    Since $\sum_{i=3}^n(i\cdot C_{t,i})\ge\sqrt{n}$, we obtain at least $\frac{1}{5}\sqrt{n}$ pairwise disjoint sets $T_1,T_2,\ldots,T_k$ of size~$3$ such that each~$T_i$ is a subset of some part $Z\in\Pi_t$.
    We denote by~$V_i$ the part of~$\Pi_t$ such that $T_i \subseteq V_i$.
    Let~$\Pi$ be the set of all parts of~$\Pi_t$ of size at most~$2$, let $\mathcal{A}:=\bigcup_{S\in\Pi}S$, and let $\mathcal{B}:=\bigcup_{i=1}^kT_i$.
    Let $H:=G[\mathcal{A},\mathcal{B}]/(\Pi\cup\{T_i:i\in[k]\})$; see Figure~\ref{fig:triplets}.
    Since $C_{t,2}\le(\alpha-\delta)n$, by Lemma~\ref{lem:pq}\ref{item:pq-1} and the definitions of $\alpha$ and~$c$,
    we have for each $i\in[k]$ that
    \begin{linenomath*}\begin{align*}
        &\E[\rdeg_{H}(T_i)]\\
        &= (1-p^3-q^3) C_{t,1}+(1-p^6-q^6) C_{t,2}\\
        &> (1-p^3-q^3)(n-2C_{t,2})+(1-p^6-q^6)C_{t,2}-
        (1-p^3-q^3)(\sqrt{n}+4)\\
        &\ge3pqn-(2(1-p^3-q^3)-(1-p^6-q^6))(\alpha-\delta)n-3pq(\sqrt{n}+4)\\
        &=2pqn+(2(1-p^3-q^3)-(1-p^6-q^6))\delta n-3pq(\sqrt{n}+4)\\
        &\ge(2pq+1.01c)n -3pq(\sqrt{n}+4).
    \end{align*}\end{linenomath*}
    Observe that 
    $3pq(\sqrt{n}+4)\le 0.005cn$ for all sufficiently large~$n$, which implies that 
    $\E[\rdeg_{H}(T_i)]\ge(2pq+1.005c)n$.
    Obviously, $\rdeg_H(T_i)<n$.
    By Lemma~\ref{lem:Chernoffbound}\ref{cond:chernoff3},
    \begin{linenomath*}\begin{align*}
        \lefteqn{\Pr[\rdeg_{H}(T_i)\le(2pq+c)n]} \\
        &\le\Pr\left[\rdeg_{H}(T_i)\le E[\rdeg_{H}(T_i)]-\frac{cn}{200}\right]
        \le\exp\left(-\frac{(\frac{1}{200}cn)^2}{2\E[\rdeg_H(T_i)]}\right)
        \le\exp\left(-\frac{c^2 n}{2\cdot200^2}\right)
    \end{align*}\end{linenomath*}
    for all sufficiently large~$n$.
    Note that $\rdeg_{G/\Pi_t}(V_i)\ge\rdeg_{H}(T_i)$ and $k\ge\frac{1}{5}\sqrt{n}$.
    Therefore,
    \begin{linenomath*}\begin{align*}
        \Pr[\DeltaR(G/\Pi_t)\le(2pq+c)n]
        &\le\prod_{i=1}^k \Pr[\rdeg_{H}(T_i)\le(2pq+c)n]\\
        &\le\exp\left(-\frac{kc^2n}{2\cdot 200^2}\right)\le\exp\left(-\frac{c^2n^{3/2}}{400000}\right),
    \end{align*}\end{linenomath*}
    which implies \eqref{eq:stars}.

    \medskip
    \noindent\textbf{Case II.} $C_{t,2}>(\alpha-\delta)n$.
    
    Let $t'\in[n]$ be the smallest integer such that $\sum_{i=4}^n(i\cdot C_{t',i})\ge2\sqrt{n}$.
    Note that $t' \ge t$.
    Similar to the argument for $\sum_{i=3}^{n}i\cdot C_{t,i}$, we have that
    \begin{linenomath*}\[
        \sum_{i=4}^n(i\cdot C_{t',i})\le\sum_{i=4}^n (i \cdot C_{t'-1,i})+6<2\sqrt{n}+6.
    \]\end{linenomath*}
    For nonnegative integers $i$ and $j$, let $\Lambda_{i,j}$ be the number of parts $Z$ of $\Pi_{t'}$ such that there are distinct $i+j$ parts $X_1,\ldots,X_i,Y_1,\ldots,Y_j\in \Pi_t$ satisfying the following two conditions.    
    \begin{itemize}
        \item $Z=X_1\cup \ldots \cup X_i \cup Y_1\cup \cdots \cup Y_j$,
        \item $\abs{X_1}=\cdots=\abs{X_i}=1$ and $\abs{Y_1}=\cdots=\abs{Y_j}=2$.
    \end{itemize}
    For $i\in\{1,2,3\}$, 
    let $C_i^*$ be the number of parts of $\Pi_t$ of size $i$
    that is a subset of a part of size at least $4$ in $\Pi_{t'}$.
    Then
    \begin{linenomath*}\begin{align*}
        C_1^*+2C_2^*+3C_3^*&\le\sum_{i=4}^n (i\cdot C_{t',i})<2\sqrt{n}+6,\\ 
        C_{t',1} &=
        C_{t,1}-(\Lambda_{1,1}+3\Lambda_{3,0}+2\Lambda_{2,0})-C_1^*, \\
        C_{t',2} &=
        C_{t,2}-\Lambda_{1,1}+\Lambda_{2,0}-C_2^*, \\
        C_{t',3} &=
        C_{t,3}+\Lambda_{1,1}+\Lambda_{3,0}-C_3^*.
    \end{align*}\end{linenomath*}
    For each part $Z$ of $\Pi_{t'}$ of size at least $4$, we partition $Z$ into parts of size~$4$ possibly except one.
    Since $\sum_{i=4}^n(i\cdot C_{t',i})\ge2\sqrt{n}$, we have at least $\frac{2}{7}\sqrt{n}$ pairwise disjoint sets $Q_1,\ldots,Q_{k'}$ of size $4$ such that each $Q_i$ is contained in some part, say $W_i$, of $\Pi_{t'}$.
    Let $\Pi':=\{Z\in\Pi_{t'}:\abs{Z}\le3\}$, $\mathcal{A}':=\bigcup_{S\in\Pi'} S$, and $\mathcal{B}':=\bigcup_{i=1}^{k'} Q_i$.
    Let $H' := G[\mathcal{A}',\mathcal{B}']/(\Pi' \cup \{Q_i : i\in[k']\})$.
    
    For a positive integer $i$, let $f_i:=1-p^i-q^i$.
    By Lemma~\ref{lem:pqadd},
    \begin{linenomath*}\begin{equation*}\label{eqn:a}
        f_4+f_8=(1-p^4-q^4)+(1-p^8-q^8)>1-p^{12}-q^{12}=f_{12}.
    \end{equation*}\end{linenomath*}
    Because $p^{4}+p^{12} \ge2p^8$ and similarly for $q$, we have
    \begin{linenomath*}\begin{equation}\label{eqn:b}
        2f_8=2(1-p^8-q^8)\ge(1-p^4-q^4)+(1-p^{12}-q^{12})=f_4 +f_{12}.
    \end{equation}\end{linenomath*}
    Recall that
    $C_{t,1}>n-2C_{t,2}-\sqrt{n}-4$.

    By the linearity of the expectation, the expected red-degree of $Q_i$ in $H'$ is
    \begin{linenomath*}\begin{align*}
    \lefteqn{
        C_{t',1} f_{4}+C_{t',2} f_{8}+C_{t',3} f_{12} }\\
        &=(C_{t,1}-(\Lambda_{1,1}+2\Lambda_{3,0}+\Lambda_{2,0})-(\Lambda_{3,0} +\Lambda_{2,0})-C_1^*)f_{4}\\
        &\qquad+(C_{t,2}-(\Lambda_{1,1}+2\Lambda_{3,0}+\Lambda_{2,0})+2(\Lambda_{3,0}+\Lambda_{2,0})-C_2^*)f_{8}\\
        &\qquad+(C_{t,3}+(\Lambda_{1,1}+2\Lambda_{3,0}+\Lambda_{2,0})-(\Lambda_{3,0}+\Lambda_{2,0})-C_3^*) f_{12}\\
        &\ge(C_{t,1}-(\Lambda_{1,1}+2\Lambda_{3,0}+\Lambda_{2,0})-C_1^* )f_{4}\\
        &\qquad+(C_{t,2}-(\Lambda_{1,1}+2\Lambda_{3,0}+\Lambda_{2,0})-C_2^*)f_{8}\\
        &\qquad+(\Lambda_{1,1}+2\Lambda_{3,0}+\Lambda_{2,0}) f_{12}\tag*{by \eqref{eqn:b} and $C_{t,3}\ge C_3^*$}\\
        &\ge(C_{t,2}-C_{t,1})f_{8}+C_{t,1}f_{12}-(C_1^*+C_2^*)\tag*{by $f_4>f_{12}-f_8$}\\
        &>(C_{t,2}-(n-2C_{t,2}-\sqrt{n}))f_{8}\tag*{because $C_{t,1}+2C_{t,2}\le n-\sqrt{n}$}\\
        &\qquad+(n-2C_{t,2}-\sqrt{n}-4)f_{12}-(2\sqrt{n}+6)\tag*{because $C_{t,1}>n-2C_{t,2}-\sqrt{n}-4$}\\
        &\ge(3C_{t,2}-n)f_8+(n-2C_{t,2})f_{12}-(3\sqrt{n}+10)\tag*{because $f_{12}<1$}\\
        &\ge(3(\alpha-\delta)-1)nf_{8}+(1-2(\alpha-\delta))nf_{12}-(3\sqrt{n}+10)\tag*{because $3f_8-2f_{12}>0$}\\
        &=(3\alpha-1)nf_{8}+(1-2\alpha)nf_{12} -( 3f_{8}-2f_{12} )\delta n-(3\sqrt{n}+10)\\
        &\ge(2pq+2\varepsilon)n-\varepsilon n -(3\sqrt{n}+10)\tag*{by the definitions of $\varepsilon$ and $\delta$}\\
        &\ge(2pq+1.01c)n-(3\sqrt{n}+10).
    \end{align*}\end{linenomath*}
    For all sufficiently large $n$, 
    $3\sqrt{n}+10 <0.005cn$,
    and therefore $\E[\rdeg_{H'}(Q_i)]\ge(2pq+1.005c)n$.
    Note that $\rdeg_{H'}(Q_i)<n$.
    By Lemma~\ref{lem:Chernoffbound}\ref{cond:chernoff3}, for all sufficiently large~$n$,
    \begin{linenomath*}\[
        \Pr\left[\rdeg_{H'}(Q_i)\le(2pq+c)n\right]
        \le\exp\left(-\frac{(\frac{1}{200}cn)^2}{2\E[\rdeg_{H'}(Q_i)]}\right)
        \le\exp\left(-\frac{c^2 n}{80000}\right).
    \]\end{linenomath*}
    Since $k'\ge\frac{2}{7}\sqrt{n}$, we have that $\Pr[\DeltaR(G/\Pi_{t'})\le(2pq+c)n]\le\exp(-\frac{k'c^2n}{80000})\le\exp(-\frac{c^2 n^{3/2}}{280000})$.
    Therefore, \eqref{eq:stars} holds.
\end{proof}

In Case II of the above proof, since $(\alpha-\delta)n < C_{t,2}$ and $\sum_{i=3}^n(i\cdot C_{t,i})<\sqrt{n}+4$, we can see that most of the first $\alpha n$ contractions are accounted for by contracting many disjoint pairs of vertices.
In this sense, the contraction sequence is quite similar to Steps~1 and~2 of the contraction sequence that we outlined in Section~\ref{sec:upper}, up until this point.
Furthermore, we remark that the expected maximum red-degree of~$G/\Pi_{t'}$ is minimized when~$C_{t',3}$ is maximized, according to the computation for $\rdeg_{H'}(Q_i)$.
The optimal strategy in order to maximize $C_{t',3}$ would be to follow Steps~3 and~4 of the contraction sequence in Section~\ref{sec:upper} as closely as possible.
This provides some intuition as to why the expected red-degree in Step~4 of the contraction sequence in Section~\ref{sec:upper} is so closely related to the expected red-degree after~$t'$ contractions in the above proof, and why the threshold of~$p^*$ arises in each case.

\subsection{When {\boldmath$p=o(1)$}}\label{subsec:sparse}

Before proving Theorem~\ref{thm:main3}, we give a simple argument to show that for $\omega(1/n)\le p\le1/2$, with high probability the twin-width of the random graph $G(n,p)$ tends to infinity as $n$ tends to infinity.

\begin{proposition}\label{prop:not tiny 1}
    For $p:=p(n)$ with $1/n\le p\le1/2$ and every $\delta\in(0,4/7)$, 
    with high probability,
    \begin{linenomath*}\[
        \tww(G(n,p))\ge(1-\delta)np-\frac{4(1-\delta)}{\delta}.
    \]\end{linenomath*}
\end{proposition}
\begin{proof}
    Let $c$ be a constant larger than $4/\delta^2$.
    We may assume that $n$ is sufficiently large so that $n\ge\frac{2ec}{\delta}\ln n$ holds.
    We may further assume that $np>\frac{4}{\delta}$ because the twin-width is always nonnegative.
    Let $G:=G(n,p)$.
    We consider two cases by comparing $p$ and $c\frac{\ln n}{n}$.

    \medskip
    \noindent\textbf{Case I.} $p\le c\frac{\ln n}{n}$.
    
    Since the expected number of edges is $p\binom{n}{2}$, by Lemma~\ref{lem:Chernoffbound}\ref{cond:chernoff3} applied to $\delta/2$, we have 
    \begin{linenomath*}\[
        \Pr\left[\abs{E(G)}<(1-\delta/2)\cdot p\binom{n}{2}\right]
        \le\exp\left(-\frac{\delta^2}{8}\cdot p\binom{n}{2}\right)=\exp(-\Omega(n^2p)),
    \]\end{linenomath*}
    which goes to $0$ as $n\to\infty$, because $p\ge1/n$.

    Let $m:=\lceil\frac{\delta np}{2}\rceil$, which is at least~$3$ as $np>4/\delta$.
    We claim that $G$ has no subgraph isomorphic to $K_{2,m}$ with high probability.
    The expected number of copies of $K_{2,m}$ in $G$ is
    \begin{linenomath*}\begin{align*}
       \binom{n}{m}\binom{n-m}{2} p^{2m}
        &\le\left(\frac{en}{m}\right)^m \frac{(n-m)^2}{2} p^{2m}\tag*{by Lemma~\ref{lem:Stirling}}\\
        &\le\left(\frac{2en}{\delta np}\right)^m\frac{(n-m)^2}{2} p^{2m}\\
        &\le n^2 \left(2ep/\delta \right)^m\\
        &\le n^2\left(\frac{2ec}{\delta}\frac{\ln n}{n}\right)^m\\
        &\le n^2 \left(\frac{2ec}{\delta} \frac{\ln n}{n}\right)^{3}\tag*{because $n \ge\frac{2ec}{\delta}\ln n$}\\
        &=\left(\frac{2ec}{\delta}\ln n\right)^{3} \frac{1}{n},
    \end{align*}\end{linenomath*}
    which converges to $0$ as $n\to\infty$.
    By Markov's inequality (Lemma~\ref{markov}), $G$~has no copy of~$K_{2,m}$ with high probability.
    Thus, with high probability,
    \begin{itemize}
        \item $\abs{E(G)}\ge(1-\delta/2)p\binom{n}{2}$ and
        \item for every pair $\{u,v\}$ of distinct vertices of $G$, $u$ and $v$ have less than $\frac{\delta}{2}np$ common neighbors in $G$.
    \end{itemize}
    By Proposition~\ref{prop:dense subgraph}, with high probability, $G$ has an induced subgraph $H$ such that
    \begin{itemize}
        \item $\delta(H)\ge\frac12 (1-\delta/2)(n-1)p$ and
        \item for every pair $\{u,v\}$ of distinct vertices of $H$, $u$ and $v$ have less than $\frac{\delta}{2}np$ common neighbors in $H$.
    \end{itemize}
    Note that for distinct vertices~$u$ and~$v$ of~$H$, 
    \begin{linenomath*}\[
        r_H(u,v)\ge(\deg_H(u)-1)+(\deg_H(v)-1)-\abs{N_H(u)\cap N_H(v)}>(1-\delta/2)(n-1)p-2-\frac{\delta np}{2}.
    \]\end{linenomath*}
    Therefore, with high probability,
    \begin{linenomath*}\[
        \tww(G)\ge\tww(H)\ge(1-\delta)np-(1-\delta/2)p-2\ge(1-\delta)np-3\ge(1-\delta)np-\frac{4(1-\delta)}{\delta},
    \]\end{linenomath*}
    where the last inequality is by $\delta\le4/7$.
    Thus, $\tww(G)\ge(1-\delta)np-\frac{4(1-\delta)}{\delta}$ with high probability when $p\le c\frac{\ln n}{n}$.

    \medskip
    \noindent\textbf{Case II.} $p>c\frac{\ln n}{n}$.
    
    Note that for distinct vertices $u$ and $v$ of $G$, $\E[r_G(u,v)]=2p(1-p)(n-2)$.
    By Lemma~\ref{lem:Chernoffbound}\ref{cond:chernoff3} applied to $\delta$,
    \begin{linenomath*}\begin{align*}
        \lefteqn{\Pr\left[|r_G(u,v)| \le(1-\delta)2p(1-p)(n-2)\right] }\\
        &\le\exp\left(-\frac{\delta^2}{2}2p(1-p)(n-2)\right)\\
        &\le\exp\left(-\frac{\delta^2}{2}2p(1-p)n+\frac{\delta^2}{2}\right)
        \tag*{because $p(1-p)\le1/4$}\\
        &\le\exp\left(-\delta^2 \cdot \frac{c\ln n}{n}\cdot\frac{1}{2}\cdot n\right)\exp(\delta^2/2)=\exp(\delta^2/2)\cdot n^{-c\delta^2/2}.
    \end{align*}\end{linenomath*}
    By the union bound, 
    \begin{linenomath*}\begin{multline*}
        \Pr\left[\text{there exist $u,v\in V(G)$ such that }\abs{r_G(u,v)}\le(1-\delta)2p(1-p)(n-2)\right]\\
        \le\binom{n}{2}\exp(\delta^2/2)\cdot n^{-c\delta^2/2},
    \end{multline*}\end{linenomath*}
    which converges to~$0$ as $n\to\infty$.
    Thus, with high probability, for every pair $\{u,v\}$ of distinct vertices of~$G$, we have 
    \begin{linenomath*}\[
        r_G(u,v)>(1-\delta)2p(1-p)(n-2)\ge(1-\delta)2p(1-p)n-(1-\delta)\ge(1-\delta)np-1.
    \]\end{linenomath*}
    Hence, $\tww(G)>(1-\delta)np-1$ with high probability when $p>c\frac{\ln n}{n}$.
    This completes the proof.
\end{proof}

We now prove Theorem~\ref{thm:main3}, which shows that the twin-width of the random graphs $G(n,p)$ for $p=\omega\left(\frac{\ln n}{n}\right)$ and $p\le1/2$ is $\Theta(n\sqrt{p})$.
We use the following theorem to obtain the upper bound.

\begin{theorem}[Ahn, Hendrey, Kim, and Oum~\cite{AHKO2021}]\label{thm:edge bound}
    For an $m$-edge graph~$G$,
    \begin{linenomath*}\[
        \tww(G)<\sqrt{3m}+\frac{m^{1/4}\sqrt{\ln m}}{4\cdot3^{1/4}}+\frac{3\cdot m^{1/4}}{2},
    \]\end{linenomath*}
    that is, $\tww(G)=O(\sqrt{m})$.
\end{theorem}
\begin{proof}[Proof of Theorem~\ref{thm:main3}]
    Let $G:=G(n,p)$.
    Since $\E[\abs{E(G)}]=p\binom{n}{2}$, by Lemma~\ref{lem:Chernoffbound}\ref{cond:chernoff1} applied to~$\delta=1$,
    \begin{linenomath*}\[
        \Pr\left[\abs{E(G)}\ge2p\binom{n}{2}\right]
        \le\exp\left(-\frac{1}{3}\cdot p\binom{n}{2}\right)=\exp\left(-\Omega(n^2p)\right).
    \]\end{linenomath*}
    Thus, by Theorem~\ref{thm:edge bound}, with high probability, $\tww(G)=O(n\sqrt{p})$.
 
    We now show that $\tww(G)>\frac{999}{10^6}n\sqrt{p}$ with high probability.
    We may assume that $n\ge n\sqrt{p}\ge15000$.
    We take an arbitrary partition sequence $\Pi_1,\Pi_2,\ldots,\Pi_n$ of $[n]$.
    For each $s\in[n]$ and an integer~$i$, let $C_{s,i}$ be the number of parts of~$\Pi_s$ of size~$i$.
    
    First we will prove an upper bound on the probability that $\DeltaR(G/\Pi_t)\le\frac{999}{10^6}n\sqrt{p}$ for all $t\in[n]$.
    Since $3/\sqrt{p}\le n$, there exists the smallest integer $s$ such that
    \(
        \sum_{i= \lceil \frac{3}{\sqrt{p}}\rceil }^n (i\cdot C_{s,i})\ge\frac{n}{2}
    \).
    Since~$s$ is the smallest, 
    \begin{linenomath*}\[ 
        \sum_{i= \lceil \frac{3}{\sqrt{p}}\rceil }^n (i\cdot C_{s,i})
        \le\sum_{i= \lceil \frac{3}{\sqrt{p}}\rceil }^n (i\cdot C_{s-1,i})+ 2\cdot\frac{3}{\sqrt{p}}
        <\frac{n}{2}+\frac{6}{\sqrt{p}},
    \]\end{linenomath*}
    and so $\sum_{i=1}^{\lceil 3/\sqrt{p}\rceil-1 }(i\cdot C_{s,i})>\frac{n}{2}-\frac{6}{\sqrt{p}}$.
     
    For each part $Z$ of $\Pi_s$ of size at least $3/\sqrt{p}$, we partition~$Z$ into $\left\lfloor \frac{|Z|}{\lceil 3/\sqrt{p}\rceil}\right\rfloor$ parts of size $\lceil3/\sqrt{p}\rceil$ and at most one part of smaller size.
    Let~$\mathcal{T}$ be the set of all parts of size $\lceil3/\sqrt{p}\rceil$ across all of these partitions.
    Then
    \begin{linenomath*}\[
        \abs{\mathcal T} \ge\frac{1}{2 \lceil 3/\sqrt{p}\rceil-1} \sum_{i= \lceil \frac{3}{\sqrt{p}}\rceil }^n (i \cdot C_{s,i})\ge 
        \frac{1}{2\left(\frac{3}{\sqrt{p}}+1\right)-1}
        \cdot\frac{n}{2}\ge 
        \frac{1}{\frac{6}{\sqrt{p}}+\frac{1}{\sqrt{2p}}}
        \cdot \frac{n}{2}
        =\frac{1}{12+\sqrt{2}}n\sqrt{p}.
    \]\end{linenomath*}
    Let $\mathcal{S}$ be a maximal set of pairwise disjoint subsets of~$V(G)$ such that each $Z\in\mathcal{S}$ is a union of parts of~$\Pi_s$ and $\frac{3}{2\sqrt{p}}\le\abs{Z}<\frac{3}{\sqrt{p}}$.
    Since $n\sqrt{p}\ge15000$, we have 
    \begin{linenomath*}\begin{align*}
        \abs{\mathcal S}
        &\ge\frac{1}{3/\sqrt{p}}\left( \sum_{i=1}^{\lceil\frac{3}{\sqrt{p}}\rceil-1} (i \cdot C_{s,i})-\frac{3}{2\sqrt{p}}\right)>\frac{\sqrt p}{3} \left( \frac{n}{2}-\frac{6}{\sqrt{p}}-\frac{3}{2\sqrt{p}}\right)\\
        &=\frac{n\sqrt{p}}{6}-\frac{5}{2}\ge\frac{n\sqrt{p}}{6}-\frac{5}{2}\cdot \frac{n\sqrt{p}}{15000}=\frac{999}{6000}n\sqrt{p}.
    \end{align*}\end{linenomath*}
    
    We construct an auxiliary random bipartite graph $H_1$ on bipartition $\mathcal{T}\cup\mathcal{S}$ such that $T\in\mathcal T$ and $S\in\mathcal S$ are adjacent in $H_1$ if $G$ has an edge having one end in $T$ and another end in $S$.
    Let $a:=\min\{\Pr[TS\in E(H_1)]:T\in\mathcal{T},S\in\mathcal{S}\}=1-\max\{(1-p)^{\abs{T}\cdot\abs{S}}:T\in\mathcal{T},S\in\mathcal{S}\}$.
    Then
    \begin{linenomath*}\[
        1-a=\max_{T\in\mathcal{T},S\in\mathcal{S}}(1-p)^{\abs{T}\cdot\abs{S}}\le(1-p)^{\frac{9}{2p}}\le\left(e^{-p}\right)^{\frac{9}{2p}}=e^{-9/2}.
    \]\end{linenomath*}
    Therefore, $a\ge1-e^{-9/2}$.
    By our choice of $a$, we have $\E[\abs{E(H_1)}]\ge a\abs{\mathcal{T}}\abs{\mathcal{S}}$.
    By Lemma~\ref{lem:Chernoffbound}\ref{cond:chernoff3} applied to~$\delta=67/100$,
    \begin{linenomath*}\[
        \Pr\left[\abs{E(H_1)} \le\frac{33a}{100}\abs{\mathcal T}\abs{\mathcal S}\right]
        \le\exp\left( -\frac{a}{2}\left(\frac{67}{100}\right)^2\abs{\mathcal T}\abs{\mathcal S}\right).
    \]\end{linenomath*}
    
    For each $S\in\mathcal{S}$, we fix a vertex $w_S\in S$.
    Let $H_2$ be the random bipartite graph on bipartition $\mathcal{T}\cup\mathcal{S}$ such that $T\in\mathcal{T}$ and $S\in\mathcal{S}$ are adjacent in~$H_2$ if and only if $T\subseteq N_G(w_S)$.
    Let $X:=\abs{E(H_2)}$.
    For each $T\in\mathcal T$ and $S\in\mathcal{S}$, we have that
    \( 
        \Pr[T\subseteq N_G(w)]=p^{\abs{T}}\le p^{5}
    \)
    as $\abs{T}\ge\lceil\frac{3}{\sqrt{p}}\rceil\ge5$.
    Thus,
    \begin{linenomath*}\[
        \E[X]=\sum_{(T,S)\in\mathcal{T}\times\mathcal{S}}p^{\abs{T}}\le\abs{\mathcal T}\abs{\mathcal S}p^{5}.
    \]\end{linenomath*}
    Since $a\ge1-e^{-9/2}$, we have that $\frac{a}{2\ln 2}\left(\frac{67}{100}\right)^2>6(1/2)^5\ge 6\cdot p^{5}$, and therefore we deduce by Lemma~\ref{lem:Chernoffbound}\ref{cond:chernoff2} that
    \begin{linenomath*}\[
        \Pr\left[X\ge\frac{a}{2\ln 2}\left(\frac{67}{100}\right)^2\abs{\mathcal T}\abs{\mathcal S}\right]
        \le2^{-\frac{a}{2\ln2}\left(\frac{67}{100}\right)^2\abs{\mathcal T}\abs{\mathcal S}}
        =\exp\left(-\frac{a}{2}\left(\frac{67}{100}\right)^2\abs{\mathcal T}\abs{\mathcal S}\right).
    \]\end{linenomath*}
    
    Now, we let $H$ be the random bipartite graph $(\mathcal{T}\cup\mathcal{S},E(H_1)\setminus E(H_2))$.
    Note that
    \begin{linenomath*}\[
        \left(\frac{33}{100}-\frac{1}{2\ln 2}\left(\frac{67}{100}\right)^2\right)a>\frac{3}{500}
    \]\end{linenomath*}
    It follows that $\abs{E(H)}\ge\abs{E(H_1)}-\abs{E(H_2)}>\frac{3}{500}\abs{\mathcal T}\abs{\mathcal S}$ with probability at least
    \begin{linenomath*}\[ 
        1-2\exp\left(-\frac{a}{2}\left(\frac{67}{100}\right)^2\abs{\mathcal T}\abs{\mathcal S}\right).
    \]\end{linenomath*}
    If $\abs{E(H)}>\frac{3}{500}\abs{\mathcal T}\abs{\mathcal S}$, then $H$ has a vertex $T\in\mathcal T$ of degree more than \(\frac3{500}\abs{\mathcal S}\ge\frac{999}{10^6}n\sqrt{p}\), and therefore 
    for the part $Z$ of $\Pi_s$ containing $T$, we have $\rdeg_{G/\Pi_s}(Z)\ge\deg_H (T)>\frac{999}{10^6}n\sqrt{p}$
    with probability at least $1-2\exp\left(-\frac{a}{2}\left(\frac{67}{100}\right)^2\abs{\mathcal T}\abs{\mathcal S}\right)$.
    This proves that 
    \begin{linenomath*}\begin{align*}
        \Pr\left[\DeltaR(G/\Pi_t)\le\frac{999}{10^6}n\sqrt{p}\text{ for all }t\in[n]\right]
        &\le2\exp\left(-\frac{a}{2}\left(\frac{67}{100}\right)^2\abs{\mathcal T}\abs{\mathcal S}\right)\\
        &\le2\exp\left(-\frac{a}{2}\left(\frac{67}{100}\right)^2\frac{999}{(12+\sqrt{2})\cdot 6000}n^2p\right).
    \end{align*}\end{linenomath*}
    
    Recall that we assume $p\ge 726\frac{\ln n}{n}$.
    Since the number of partition sequences of~$[n]$ is at most $\exp(2n\ln n)$ by Lemma~\ref{lem:number-partition-sequence}, the probability that there is a $(\frac{999}{10^6}n\sqrt{p})$-contraction sequence is at most 
    \begin{linenomath*}\begin{multline*}
        2\exp\left(-\frac{a}{2}\left(\frac{67}{100}\right)^2\frac{999}{(12+\sqrt{2})\cdot 6000}n^2p\right)\cdot e^{2n\ln n}\\
        \le2\exp\left(\left(2-\frac{726(1-e^{-9/2})\cdot 999\cdot 0.67^2}{2\cdot(12+\sqrt2)\cdot 6000}\right)n\ln n\right)\le2 e^{-0.0001n\ln n},
    \end{multline*}\end{linenomath*}
    which converges to $0$ as $n\to \infty$.
    This completes the proof.
 \end{proof}
 
\section*{Acknowledgment}
We thank the anonymous referees for their careful reading and helpful suggestions.

\providecommand{\bysame}{\leavevmode\hbox to3em{\hrulefill}\thinspace}
\providecommand{\MR}{\relax\ifhmode\unskip\space\fi MR }
\providecommand{\MRhref}[2]{%
  \href{http://www.ams.org/mathscinet-getitem?mr=#1}{#2}
}
\providecommand{\href}[2]{#2}

\appendix

\section{Proofs of Lemmas~\ref{lem:binomUpper} and~\ref{lem:binomLower}}\label{app:proofs}

We will use the following lemmas to prove Lemma~\ref{lem:binomUpper}

\begin{lemma}[Hoeffding~\cite{Hoeffding1963}]\label{lem:sumBernoulli}
    Let $X_1,\ldots,X_n$ be mutually independent random variables such that $0\le X_i \le1$ for each $i\in\{1,\ldots,n\}$.
    Let $X:=\sum_{i=1}^nX_i$ and $p:=\frac{1}{n}\E[X]$.
    Then for all $t\in(p,1)$,
    \begin{linenomath*}\[
        \Pr[X \ge tn]\le e^{-n \cdot D(t\| p)}
    \]\end{linenomath*}
    where $D(x\|y) := x \ln \frac{x}{y}+(1-x)\ln \frac{1-x}{1-y}$.
\end{lemma}

\begin{lemma}\label{lem:approx ln}
    For all $x\in[0,1)$ and all $y\ge0$,
    \begin{linenomath*}\begin{align*}
        \ln (1-x)&\le -x-\frac{x^2}{2},\\
        y-\frac{y^2}{2}\le\ln(1+y) &\le 
        y-\frac{y^2}{2}+\frac{y^3}{3}.
    \end{align*}\end{linenomath*}
    Moreover, if $0\le x\le 3/10$, then
    \begin{linenomath*}\[
        \ln (1-x)\ge -x-\frac{x^2}{2}-x^3.
    \]\end{linenomath*}
\end{lemma}
\begin{proof}
    We may assume that neither $x$ nor $y$ is $0$.
    By Taylor's theorem, there exist $c\in(0,x)$ and $c'\in(0,y)$ such that 
    \begin{linenomath*}\[ 
        \ln(1-x)=-x-\frac12 x^2-\frac{1}{3(1-c)^3}{x^3},\quad 
        \ln(1+y)=y-\frac1{2} y^2+\frac{1}{3(1+c')^3}{y^3}.
    \]\end{linenomath*}
    Since $c\in(0,x)$, if $x\le 3/10$, then $\frac{1}{3(1-c)^3}< 1000/(3\cdot7^3)<1$.
    Thus, the inequalities hold.
\end{proof}

We now prove Lemma~\ref{lem:binomUpper}.

\binomUpper*

\begin{proof}
    Denote $q:=1-p$.
    Let $Y_i=1-X_i$ and $Y=\sum_{i=1}^n Y_i$.
    Then $X\le(p-\varepsilon)n$ if and only if $Y\ge(q+\varepsilon)n$.
    We claim that
    \begin{linenomath*}\[ 
        \Pr[Y \ge(q+\varepsilon)n] \le\exp\left(-\frac{n \varepsilon^2}{2pq}+\frac{n \varepsilon^3}{2p^2q^2} \right).
    \]\end{linenomath*}
    Note that $0<\varepsilon/p\le3/10<1$.
    Thus, by Lemma~\ref{lem:approx ln},
    \begin{linenomath*}\begin{align*}
        D(q+\varepsilon\| q)
        &=(q+\varepsilon)\ln\left(1+\frac{\varepsilon}{q}\right)+(p-\varepsilon)\ln\left(1-\frac{\varepsilon}{p}\right)\\
        &\ge(q+\varepsilon)\left(\frac{\varepsilon}{q}-\frac{\varepsilon^2}{2q^2}\right)
        +(p-\varepsilon)\left(-\frac{\varepsilon}{p}-\frac{\varepsilon^2}{2p^2}-\frac{\varepsilon^3}{p^3}\right)\\
        &=\varepsilon^2\left(\frac{1}{2p}+\frac{1}{2q}\right)+\varepsilon^3\left(-\frac{1}{2q^2}+\frac{1}{2p^2}-\frac{1}{p^2}\right)+\frac{\varepsilon^4}{p^3}\\
        &=\frac{\varepsilon^2}{2pq}+\varepsilon^3\left(\frac{-1+2pq}{2p^2q^2}\right)+\frac{\varepsilon^4}{p^3}>\frac{\varepsilon^2}{2pq}-\frac{\varepsilon^3}{2p^2q^2}.
    \end{align*}\end{linenomath*}
    Since $0<\varepsilon\le3p/10$, we have that $q+\varepsilon\in(q,1)$.
    Hence, by Lemma~\ref{lem:sumBernoulli}, we obtain the desired inequality.
\end{proof}

To prove Lemma~\ref{lem:binomLower}, we use the following standard bounds on the binomial coefficients obtained from Stirling's formula.
In the proof of Lemma~\ref{lem:binomLower}, we only use the lower bound for $\binom{n}{\kappa n}$ in Lemma~\ref{lem:approxBC}.

\begin{lemma}[See~{\cite[Lemma~17.5.1]{Cover2006}}]\label{lem:approxBC}
    For an integer $n\ge1$ and $\kappa\in(0,1)$ such that~$\kappa n$ is an integer,
    \begin{linenomath*}\[
        \frac{1}{\sqrt{8n\kappa(1-\kappa)}} e^{n \cdot H(\kappa,(1-\kappa))} \le\binom{n}{\kappa n}\le\frac{1}{\sqrt{\pi n\kappa(1-\kappa)}} e^{n \cdot H(\kappa,(1-\kappa))}
    \]\end{linenomath*}
    where $H(x,y) :=-x \ln x-y \ln y$.
\end{lemma}

We now prove Lemma~\ref{lem:binomLower}.

\binomLower*

\begin{proof}
    Let $q:=1-p$.
    We claim that if $i$ is an integer satisfying that
    $(p-\varepsilon)n-\sqrt{n}<i\le(p-\varepsilon)n$, then
    \begin{linenomath*}\[
       \binom{n}{i}p^iq^{n-i}\ge\frac{1}{\sqrt{2n}} 
        \exp\left(
            -\frac{n\varepsilon^2}{2pq}
            -\frac{3\sqrt{n\varepsilon^2}}{2pq}
            -\frac{4 n\varepsilon^3}{p^2q^2}
        \right).
    \]\end{linenomath*}
    Since $1/\sqrt{n}\le\varepsilon\le p/2$, we have that $(p-\varepsilon)n-\sqrt{n}=(p-\varepsilon-1/\sqrt{n})n\ge0$.
    Thus, $i$ is positive.
    By Lemma~\ref{lem:approxBC} and the inequality $pq\le1/4$, we have that
    \begin{linenomath*}\[ 
       \binom{n}{i}p^iq^{n-i}\ge
        \frac{1}{\sqrt{2n}} 
        e^{n\cdot H(i/n,1-(i/n))}p^i q^{n-i}
       =\frac{1}{\sqrt{2n}} e^{-n D((i/n)\| p)}.
    \]\end{linenomath*}
    Let $\delta:=p-(i/n)$.
    Note that $\delta<p$ because $i$ is positive.
    Since $(p-\varepsilon)n-\sqrt{n}\le i\le(p-\varepsilon)n$ and $1/\sqrt{n}\le\varepsilon\le\min\{p/2,q\}$, we deduce that
    \begin{linenomath*}\[
        \varepsilon\le\delta\le\varepsilon+\frac{1}{\sqrt{n}}\le2\varepsilon\le\min\{p,2q\}.
    \]\end{linenomath*}
    Since $0<\delta/p<1$, by Lemma~\ref{lem:approx ln},
    \begin{linenomath*}\begin{align*}
        D(p-\delta \| p)
        &= (p-\delta)\ln \frac{p-\delta}{p}+(1-p+\delta)\ln \frac{1-p+\delta}{1-p}\\
        &\le
        (p-\delta)\left(-\frac{\delta}{p}-\frac{\delta^2}{2 p^2} \right)
        +
        (q+\delta)\left( \frac{\delta}{q}-\frac{\delta^2}{2q^2}
        +\frac{\delta^3}{3q^3} \right) 
        \\
        &=\delta^2 \left(\frac1{2p}+\frac1{2q}\right)+\delta^3 \left(\frac{1}{2p^2}-\frac1{2q^2}+\frac{1}{3q^2}\right)+\frac{\delta^4}{3q^3}\\
        &=\frac{\delta^2}{2pq}+\delta^3 \left( \frac{1}{2p^2}+\frac{1}{2q^2} \right)-\frac{2\delta^3}{3q^2}+\frac{\delta^4}{3q^3}\\
        &\le\frac{\delta^2}{2pq}+\frac{\delta^3}{2p^2q^2}
        &\text{because }\delta\le2q\\
        &\le\frac{\varepsilon^2}{2pq}+\frac{2\sqrt{n\varepsilon^2}+1}{2pqn}+\frac{\delta^3}{2p^2q^2}
        &\text{because }\delta\le\varepsilon+\frac{1}{\sqrt{n}}\\
        &\le\frac{\varepsilon^2}{2pq}+\frac{2\sqrt{n\varepsilon^2}+1}{2pqn}+\frac{4\varepsilon^3}{p^2q^2}
        &\text{because }\delta\le2\varepsilon\\
        &\le\frac{\varepsilon^2}{2pq}+\frac{3\sqrt{n\varepsilon^2}}{2pqn}+\frac{4\varepsilon^3}{p^2q^2}
        &\text{because }n\varepsilon^2\ge1.
    \end{align*}\end{linenomath*}
    Thus, the claim holds.

    The number of integers $i$ with $(p-\varepsilon)n-\sqrt{n}<i\le(p-\varepsilon)n$ is at least $\lfloor\sqrt{n}\rfloor>\sqrt{n}-1\ge\frac{1}{2}\sqrt{n}$ because $n\ge4$.
    Therefore, we deduce that
    \begin{linenomath*}\begin{align*}
        \Pr[\mathcal{B}(n,p)\le(p-\varepsilon)n]
        &\ge\sum_{(p-\varepsilon)n-\sqrt n<i\le(p-\varepsilon)n}\binom{n}{i} p^i q^{n-i}\\
        &\ge\frac{\sqrt{n}}{2}\cdot\frac{1}{\sqrt{2n}}\exp\left(-\frac{n\varepsilon^2}{2pq}-\frac{3\sqrt{n\varepsilon^2}}{2pq}-\frac{4 n\varepsilon^3}{p^2q^2}\right),
    \end{align*}\end{linenomath*}
    and this completes the proof.
\end{proof}

\section{Proof of Lemma~\ref{lem:ab}}\label{app:inequalities}

\begin{figure}[t]
    \centering
    \includegraphics[scale=0.18]{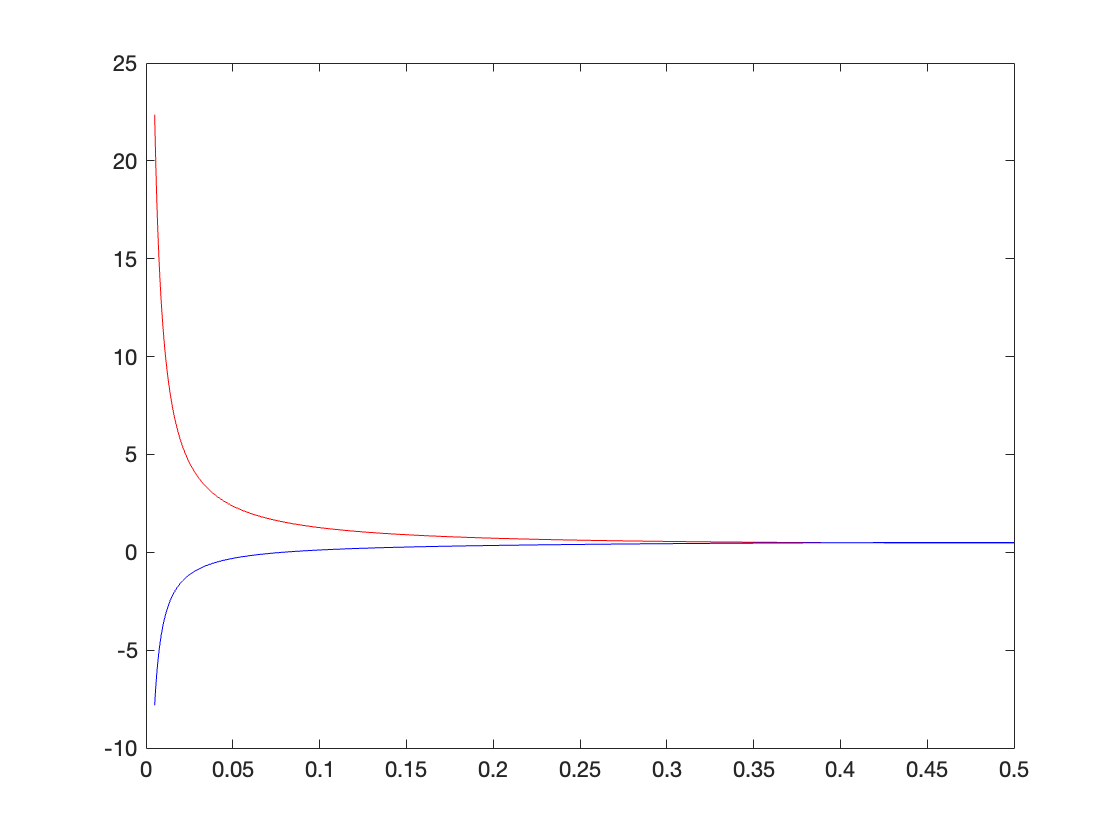}
    \hspace{0.1cm}
    \includegraphics[scale=0.18]{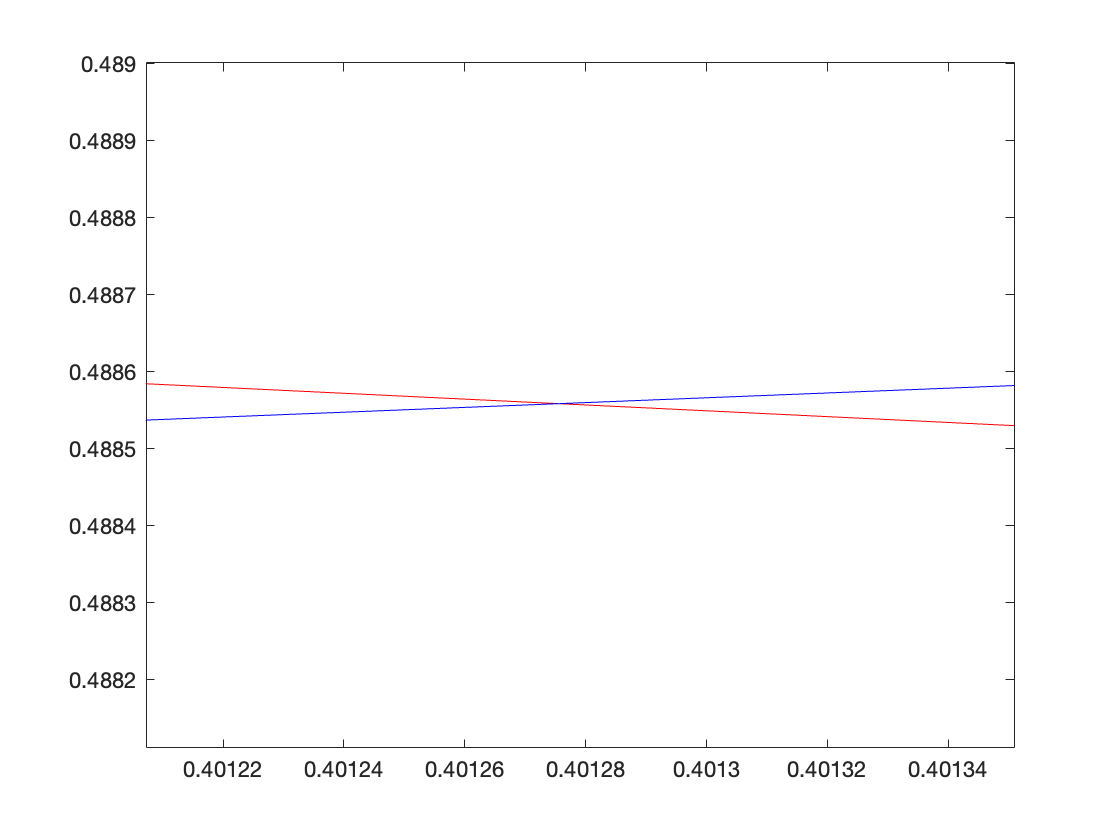}
    \caption{The red lines are graphs of $y=\alpha(x)$ and the blue lines are graphs of $y=\beta(x)$.}
    \label{fig:AlphaBeta}
\end{figure}

We provide the proof of Lemma~\ref{lem:ab}.
See Figure~\ref{fig:AlphaBeta} for graphs of $\alpha(x)$ and $\beta(x)$.

\lemab*

\begin{proof}
    (i)
    We first show that $\alpha$ is decreasing on $(0,1/2]$.
    Let $\alpha_2:(0,1)\to\mathbb{R}$ with $\alpha_2(z):=\frac{1}{z(9-2z)}$.
    For all $p\in(0,1)$ and $q:=1-p$, we have that $1-p^3-q^3=3pq$ and
    \begin{linenomath*}\begin{align*}
        1-p^6-q^6
       =\sum_{i=1}^5\binom{6}{i} p^i q^{6-i}
        &= 6pq \left( p^4+\frac{5}{2}p^3 q+\frac{10}{3}p^2 q^2+\frac{5}{2}p q^3+q^4 \right)\\
        &= 6pq \left((p+q)^4-\frac{3}{2}pq \left( p^2+\frac{16}{9}pq+q^2 \right)\right)\\
        &= 6pq \left( 1-\frac{3}{2}pq \left( 1-\frac{2}{9}pq \right)\right)\\
        &= 6pq-(pq)^2(9-2pq)=6pq-\frac{pq}{\alpha_2(pq)}.
    \end{align*}\end{linenomath*}
    Therefore, $\alpha(x)=\alpha_2(x(1-x))$.
    Note that $\alpha_2'(z)=\frac{4z-9}{z^2(9-2z)^2}<0$ for all $z\in(0,1)$.
    Since $x\in(0,1/2]$, we have that $\alpha'(x)=(1-2x)\cdot\alpha_2'(x(1-x))\le0$, where the equality holds only if $x=1/2$.
    Hence, $\alpha$ is decreasing on $(0,1/2]$.

    We now show that $\beta$ is increasing on $(0,1/2]$.
    Let $\beta_2:(0,1)\to\mathbb{R}$ with
    \begin{linenomath*}\[
        \beta_2(z) := \frac{2 z^5-36 z^4+103 z^3-96 z^2+34 z-2}{4z( z^4-18 z^3+51 z^2-44 z+12 )}.
    \]\end{linenomath*}
    Similar to the computation of $1-p^6-q^6$, we have that
    \begin{linenomath*}\begin{align*}
        1-p^8-q^8
        &= 8pq \left( 1-\frac{5}{2}pq \left( 1-\frac{4}{5}pq \left( 1-\frac{1}{8}pq \right)\right)\right), \\
        1-p^{12}-q^{12}
        &= 12pq \left( 1-\frac{9}{2}pq \left( 1-\frac{56}{27}pq \left( 1-\frac{15}{16}pq \left( 1-\frac{12}{35} pq \left( 1-\frac{1}{18}pq \right)\right)\right)\right)\right).
    \end{align*}\end{linenomath*}
    One can check that $\beta(x)=\beta_2(x(1-x))$ and
    \begin{linenomath*}\[
        \beta_2'(z)=\frac{z^7-21z^6+164z^5-457z^4+570z^3-325z^2+88z-12}{-2z^2(z^4-18z^3+51z^2-44z+12)^2}.
    \]\end{linenomath*}
    For all $z\in(0,1)$, let $k(z)$ be the numerator of $\beta_2'(z)$, that is,
    \begin{linenomath*}\[
        k(z) := z^7-21z^6+164z^5-457z^4+570z^3-325z^2+88z-12.
    \]\end{linenomath*}
    One can check the following.
    \begin{itemize}
        \item $k(1/4)$, $k''(1/4)$, and $k^{(4)}(1/4)$ are negative.
        \item $k'(1/4)$, $k^{(3)}(1/4)$, and $k^{(5)}(1/4)$ are positive.
        \item $k^{(6)}(z)=5040z-15120 <0$ for all $z \in(0,1/4]$.
    \end{itemize}
    Thus, by the Taylor expansion of $k(z)$ at $1/4$, we have that $k(z)<0$ for all $z\in(0,1/4]$, and therefore $\beta_2'(z) >0$ for all $z\in(0,1/4]$.
    Since $x\in(0,1/2]$, we have that $\beta'(x)=(1-2x)\cdot \beta'(x(1-x))\ge0$, where the equality holds only if $x=1/2$.
    Hence, $\beta$ is increasing on $(0,1/2]$.
    
    (ii)
    It can be easily checked that $\alpha(0.4012)>\beta(0.4012)$ and $\alpha(0.4013)<\beta(0.4013)$.
    Thus, (ii) is derived by (i).
    
    (iii)
    By (i) and (ii),
    $p^*-x$ have the same sign with $\alpha(x)-\beta(x)$.
    By Lemma~\ref{lem:pq}\ref{item:pq-1}, $\frac{1-x^{12}-(1-x)^{12}}{12}<\frac{1-x^8-(1-x)^8}{8}$, and therefore 
    $\alpha(x)-\beta(x)$ has the same sign with 
    \begin{linenomath*}\begin{align*}
        &(3(1-x^8-(1-x)^8)-2(1-x^{12}-(1-x)^{12}))(\alpha(x)-\beta(x))\\   
        &=(1-x^8-(1-x)^8)(3\alpha(x)-1)+(1-x^{12}-(1-x)^{12})(1-2\alpha(x))
        -2x(1-x).
    \end{align*}\end{linenomath*}
    
    (iv)
    Let $z := x(1-x)$.
    Recall that $\alpha(x)=\alpha_2(z)=\frac{1}{z(9-2z)}$.
    Then $\frac{1-\alpha(x)}{2}<2x(1-x)$ if and only if $0<1+z(4z-1)(9-2z)$, which holds for all $z\in(0,1/4]$.
    Since $x\in(0,1/2]$, we have that $z\in(0,1/4]$, and therefore (iv) holds.
\end{proof}

\end{document}